\documentclass[journal]{IEEEtran}

\usepackage{amsmath,amsthm}
\usepackage{amssymb}
\usepackage{mathrsfs}
\usepackage{amsfonts}
\usepackage{graphics}
\usepackage{float}
\usepackage{graphicx}
\usepackage{indentfirst}
\usepackage{enumerate}
\usepackage{caption}
\usepackage{cite}
\usepackage{arydshln}

\theoremstyle{definition}
\newtheorem{theorem}{Theorem}
\newtheorem{lemma}{Lemma}
\newtheorem{corollary}{Corollary}

\newtheorem{remark}{Remark}

\newtheorem{example}{Example}

\def\Ds{\displaystyle}

\makeatletter

\newcommand{\Rmnum}[1]{\expandafter\@slowromancap\romannumeral #1@}
\makeatother
\ifCLASSINFOpdf
\else
\fi

\begin{document}

\title{Distributed Averaging Problems over Directed Signed Networks}

\author{Mingjun~Du,~Deyuan~Meng,~\IEEEmembership{Senior~Member,~IEEE},~and~Zheng-Guang~Wu% <-this % stops a space

\thanks{Mingjun Du is with the School of Electrical Engineering and Automation, Qilu University of Technology (Shandong Academy of Science), Jinan, Shandong Province 250353, P. R. China. He was with the Seventh Research Division, Beihang University (BUAA), Beijing 100191, P. R. China (Email: dumingjun0421@163.com).}
\thanks{Deyuan Meng (corresponding author) is with the Seventh Research Division, Beihang University (BUAA), Beijing 100191, P. R. China, and also with the School of Automation Science and Electrical Engineering, Beihang University (BUAA), Beijing 100191, P. R. China (Email: dymeng@buaa.edu.cn).}
\thanks{Zheng-Guang. Wu is with the National Laboratory of Industrial Control Technology, Institute of Cyber-Systems and Control, Zhejiang University, Hangzhou 310027, P. R. China (Email: nashwzhg@zju.edu.cn).}}

\maketitle

\begin{abstract}
This paper aims at addressing distributed averaging problems for signed networks in the presence of general directed topologies that are represented by signed digraphs. A new class of improved Laplacian potential functions is proposed by presenting two notions of any signed digraph: induced unsigned digraph and mirror (undirected) signed graph, based on which two distributed averaging protocols are designed using the nearest neighbor rules. It is shown that with any of the designed protocols, signed-average consensus (respectively, state stability) can be achieved if and only if the associated signed digraph of signed network is structurally balanced (respectively, unbalanced), regardless of whether weight balance is satisfied or not. Further, improved Laplacian potential functions can be exploited to solve fixed-time consensus problems of signed networks with directed topologies, in which a nonlinear distributed protocol is proposed to ensure the bipartite consensus or state stability within a fixed time. Additionally, the convergence analyses of directed signed networks can be implemented with the Lyapunov stability analysis method, which is realized by revealing the tight relationship between convergence behaviors of directed signed networks and properties of improved Laplacian potential functions. Illustrative examples are presented to demonstrate the validity of our theoretical results for directed signed networks.
\end{abstract}


\IEEEpeerreviewmaketitle

\section{Introduction}

Networked systems generally consist of multiple interacting agents to deal with problems that are difficult to address for a single agent. One of the most considered classes of networked systems is called {\it unsigned networks} (also called conventional or traditional networks), where cooperative interactions among agents are only contained. To characterize unsigned networks, unsigned graphs can be conveniently employed such that their nodes and edges with positive weights are adopted to represent agents and cooperative interactions among agents, respectively. By leveraging this property, distributed control has attracted considerable attention in unsigned networks (see, e.g., \cite{Ren05,Cao08,Yu10,Yuce17,Abde18,Hong19,Qian19,Lv20}). In particular, consensus is one of the fundamental distributed control problems in unsigned networks such that all the agents cooperate with each other to accomplish a common objective.

Of specific interest in consensus is the average consensus of unsigned networks such that all agents converge to the average value of the initial states. As shown in the consensus literature, the average consensus generally can provide a solution for the distributed averaging problems on unsigned networks, which has been extensively investigated (see, e.g., \cite{Olfa04,Cai12,Prio14,Atri12,Cai14,Hadj14,Li16}). In \cite{Olfa04}, a distributed control protocol has been proposed to guarantee the average consensus of unsigned networks under unsigned digraphs that are both strongly connected and weight balanced. The average consensus problems of unsigned networks under any strongly connected topologies have been discussed in \cite{Cai12}, for which a consensus protocol is constructed by introducing a ``surplus'' variable in contrast to the protocol of \cite{Olfa04}. Regarding this problem, another new protocol has been designed by using the left eigenvector of the Laplacian matrix corresponding to the zero eigenvalue in \cite{Prio14}. For unsigned networks regardless of time-varying topologies, the average consensus problems have been explored in \cite{Atri12,Cai14}. Besides, a distributed control protocol has been proposed to guarantee the average consensus of unsigned networks with time-varying delays in \cite{Hadj14}. When considering unsigned networks subject to nonlinear dynamics, an unscented Kalman filter algorithm has been given to realize the average consensus objective in \cite{Li16}.

Recently, {\it signed networks} emerging from social networks have drawn great interests because of its potential applications (see \cite{Pros18} for more details). Different from unsigned networks, signed networks contain not only cooperative interactions but also antagonistic interactions, which should be described under signed digraphs such that the edges with positive (respectively, negative) weights can be leveraged to represent the cooperative (respectively, antagonistic) interactions among agents. Owing to the antagonistic interactions among agents, signed networks may naturally generate more plentiful collective behaviors than unsigned networks. In \cite{Alta13}, a basic framework of addressing distributed control issues on signed networks is established. It has been shown that, under strongly connected communication topologies, signed networks achieve bipartite consensus when they are associated with structurally balanced signed digraphs; and the state stability emerges, otherwise. Bipartite consensus means that all agents converge to two different values with the same modulus but opposite signs, and it includes consensus as a particular case. Additionally, not only has bipartite consensus been extended to signed networks with general linear dynamics \cite{Zhang17,Qin17,Hu18}, but also many other kinds of convergence behaviors of signed networks have been shown, such as bipartite consensus tracking \cite{Wen18,Yu19,Jiao19}, interval bipartite consensus \cite{Meng16,Xia16,Meng18}, modulus consensus \cite{Zi16,Anto16}, bipartite swarming behavior \cite{Hu19a}, and finite-time bipartite consensus \cite{Jia16,Zuo16,Zhao17,Wang18,Liu19,Lu20}.

However, there have been presented quite limited results for solving distributed averaging problems on signed networks. A main reason is that the average consensus may no longer work due to the presence of antagonistic interactions. To overcome this problem, distributed averaging of signed networks mainly concentrates on designing distributed protocols to accomplish signed-average consensus, instead of the average consensus. In \cite{Alta13}, it has been disclosed that signed networks can be driven to achieve the signed-average consensus when the associated signed digraphs are strongly connected and weight balanced. Nevertheless, it is worthwhile pointing out that the protocol of \cite{Alta13} is ineffective for the signed-average consensus once the weight balance is broken. To the best of our knowledge, there have been reported no investigations on distributed averaging problems of signed networks under any directed topologies.

In this paper, we contribute to solving distributed averaging problems for arbitrary directed signed networks in the presence of strongly connected communication topologies. We propose two notions of any signed digraph: induced unsigned digraph and mirror signed graph, with which we simultaneously introduce a new class of improved Laplacian potential functions to measure the total disagreements among agents. With the help of using improved Laplacian potential functions, we present two distributed protocols based on the nearest neighbor rules. Our two presented protocols can guarantee the signed-average consensus (respectively, state stability) of the directed signed network if and only if its associated signed digraph is structurally balanced (respectively, unbalanced), no matter whether the weight balance is satisfied or not. Besides, we can establish a relation between convergence behaviors of signed networks and improved Laplacian potential functions. This brings an advantage in developing the convergence analyses of our two presented protocols by the Lyapunov stability analysis method. In addition, we provide an application of improved Laplacian potential functions, based on which the fixed-time consensus problems can be solved for signed networks in the presence of directed topologies. We can introduce a nonlinear distributed protocol such that all agents reach the bipartite consensus (or state stability) within a fixed time. The associated convergence analyses can be developed by employing improved Laplacian potential functions. Two illustrative examples are exhibited to illustrate the effectiveness of our developed results.

The rest of paper is organized as follows. In Section \Rmnum{2}, the problem statement of distributed averaging control is provided for signed networks. In Section \Rmnum{3}, we propose some notions of signed digraphs, including induced unsigned digraph, mirror signed graph and improved Laplacian potential functions. In Section \Rmnum{4}, two distributed protocol are designed based on the nearest neighbor rules to ensure the signed-average consensus of directed signed networks. In Section \Rmnum{5}, a nonlinear protocol is proposed for directed signed networks and the associated fixed-time consensus results are derived. Simulation examples and conclusions are given in Sections \Rmnum{6} and \Rmnum{7}, respectively.

\emph{Notations}: For a positive integer $n$, $\mathcal{F}_n=\{1$, $2$, $\cdots$, $n\}$, $1_n=[1$, $1$, $\cdots$, $1]^{T}\in \mathbb{R}^n$, $0_n=[0$, $0$, $\cdots$, $0]^T\in \mathbb{R}^n$, and $\mathrm{diag}\{d_1$, $d_2$, $\cdots$, $d_n\}$ is diagonal matrix whose diagonal elements are $d_1$, $d_2$, $\cdots$, $d_n$ and non-diagonal elements are zero. For a matrix $M\in \mathbb{R}^{n\times n}$, $\det(M)$ and $\mathcal{N}(M)$ represent the determinant and null space of $M$, respectively. We denote $M>0$ (respectively, $M\geq0$) as the positive (respectively, semi-positive) definite matrix. For a real number $a\in \mathbb{R}$, let $|a|$ and $\mathrm{sgn}(a)$ represent the absolute value and the sign function of $a$, respectively. The set of all $n$-by-$n$ gauge transformations is given by
\begin{equation*}
\mathcal{D}=\{D_n=\mathrm{diag}\{\sigma_1,\sigma_2,\cdots,\sigma_n\}:\sigma_i \in \{\pm1\},i\in \mathcal{F}_n\}.
\end{equation*}

\section{Problem Statement}

Consider signed networks with a collection of $n$ agents given by $\mathcal{V}=\left\{v_{i}:i\in \mathcal{F}_{n}\right\}$. Let every agent $v_{i}$ have single-integrator dynamics described by
\begin{equation}\label{eq1}
\dot{x}_i(t)=u_i(t),\quad i\in \mathcal{F}_{n}
\end{equation}

\noindent where $x_i(t)\in \mathbb{R}$ and $u_i(t)\in \mathbb{R}$ are the information state and control protocol of $v_i$, respectively. The problem of our interest is to design protocols to achieve the distributed averaging of all agents in the presence of cooperative-antagonistic interactions among agents. For the convenience of our following analyses, we denote $x_{i}(0)\triangleq x_{i0}$, $\forall i\in\mathcal{F}_{n}$ and
\[
\aligned
x(t)&=\left[x_{1}(t),x_{2}(t),\cdots,x_{n}(t)\right]^{T}\\
x_{0}&=\left[x_{10},x_{20},\cdots,x_{n0}\right]^{T}.
\endaligned
\]

We say that the signed network (\ref{eq1}) achieves {\it signed-average consensus} if for any initial states $x_{i0}\in\mathbb{R}$, $\forall i\in\mathcal{F}_{n}$, there exist some scalars $\sigma_{i}\in\{\pm1\}$ such that
\begin{equation}\label{eq01}
\lim_{t\to\infty}x_{i}(t)=\frac{\Ds\sigma_{i}}{\Ds n}\sum_{j=1}^{n}\sigma_{j}x_{j0},\quad i\in \mathcal{F}_{n}
\end{equation}

\noindent where $\sigma_{i}$ assigns the sign to each agent $v_{i}$. It is worth pointing out that the signed-average consensus may reflect the effects of both cooperative and antagonistic interactions among agents by the selections of scalars $\sigma_{i}$, $\forall i\in\mathcal{F}_{n}$. This gives an alternative solution to the distributed averaging problem of agents involved in signed networks. In particular, if $\sigma_{i}=1$, $\forall i\in\mathcal{F}_{n}$, then the signed-average consensus collapses into the traditional average consensus of (unsigned) networks.

Another fact we need to highlight is that the signed-average consensus (\ref{eq01}) represents a specific type of {\it bipartite consensus}. In general, we say that the signed network (\ref{eq1}) reaches bipartite consensus if for any $x_{i0}\in\mathbb{R}$, $\forall i\in\mathcal{F}_{n}$, there exist some scalars $\sigma_{i}\in\{\pm1\}$ and $c\neq0$ such that (see also \cite[Definition 1]{Alta13})
\[\lim_{t\to\infty}x_{i}(t)=\sigma_{i}c,\quad i\in\mathcal{F}_{n}
\]

\noindent where $c$ depends closely on $x_{i0}$, $\forall i\in\mathcal{F}_{n}$. For signed networks, the {\it state stability} (or stability for short) is usually considered as a counterpart problem of bipartite consensus. Namely, for any $x_{i0}\in\mathbb{R}$, $\forall i\in\mathcal{F}_{n}$, the signed network (\ref{eq1}) is said to achieve state stability if
\[\lim_{t\to\infty}x_{i}(t)=0,\quad i\in\mathcal{F}_{n}.\]

Though many notable results have been derived for bipartite consensus of signed networks, none of them can be adopted to address the signed-average consensus problems in the presence of directed topologies. This issue will be solved in the current paper to achieve the distributed averaging of agents in directed signed networks. A new protocol design method for distributed averaging will be introduced simultaneously, for which an idea of a new class of improved Laplacian potential functions will be explored for a signed directed graph (digraph for short).

\begin{remark}\label{rem0}
By resorting to the state vector $x(t)$, we know that signed-average consensus (respectively, bipartite consensus) of (\ref{eq1}) refers to $\lim_{t\to\infty}x(t)=\left(1_{n}^{T}D_{n}x_{0}/n\right)D_{n}1_{n}$ (respectively, $\lim_{t\to\infty}x(t)=cD_{n}1_{n}$), where $D_{n}\in\mathcal{D}$ denotes some gauge transformation, and that stability of (\ref{eq1}) means $\lim_{t\to\infty}x(t)=0$. It has been revealed in the literature (see, e.g., \cite{Alta13}) that $D_{n}$ is generally related to the sign patterns of signed networks. In addition, the behaviors of signed networks depend heavily on their sign patterns. These will be also developed for distributed averaging of signed networks with general directed topologies.
\end{remark}

\section{Graph-Theoretic Analysis of Signed Networks}

We first introduce notions and properties related with signed digraphs, together with their induced graphs. Then we explore the Laplacian matrices of signed digraphs and their motivated Laplacian potential functions of the signed network (\ref{eq1}). In the following discussions, the time variable $t$ will be omitted for simplicity when no confusions may be caused.

\subsection{Notions and Properties of Signed Digraphs}

The interactions among agents can be modeled by a signed digraph $\mathcal{G}=(\mathcal{V},\mathcal{E},A)$ that includes a node set $\mathcal{V}=\{v_{1}$, $v_{2}$, $\cdots$, $v_{n}\}$, an edge set $\mathcal{E}\subseteq\{(v_i,v_j):v_{i},v_{j}\in\mathcal{V}\}$ and an adjacency weight matrix $A=[a_{ij}]\in\mathbb{R}^{n\times n}$ such that $a_{ij}\neq0\Leftrightarrow(v_j,v_i)\in\mathcal{E}$ and $a_{ij}=0$, otherwise. We suppose that the signed digraph $\mathcal{G}$ contains no self-loops, i.e., $a_{ii}=0$, $\forall i\in\mathcal{F}_n$, and is digon sign-symmetric, i.e., $a_{ij}a_{ji}\geq 0$, $\forall i,j\in\mathcal{F}_{n}$. For each node $v_i$, its in-degree and out-degree are defined as
\begin{equation*}
\mathrm{deg}_{\mathrm{in}}(v_{i})=\sum_{j=1}^n|a_{ij}|~\mbox{and}~\mathrm{deg}_{\mathrm{out}}(v_{i})=\sum_{j=1}^n|a_{ji}|.
\end{equation*}

\noindent We say that the signed digraph $\mathcal{G}$ is weight balanced when the condition $\mathrm{deg}_{\mathrm{out}}(v_{i})=\mathrm{deg}_{\mathrm{in}}(v_{i})$ holds for all $v_i\in \mathcal{V}$. Denote $\Delta=\mathrm{diag}\{\deg_{\mathrm{in}}(v_{1})$, $\deg_{\mathrm{in}}(v_{2})$, $\cdots$, $\deg_{\mathrm{in}}(v_{n})\}$ as the in-degree matrix of $\mathcal{G}$. The Laplacian matrix of $\mathcal{G}$, denoted by $L$, satisfies $L=\Delta-A$. In particular, when $A=A^{T}$ holds, $\mathcal{G}$ collapses into an undirected signed graph. An edge $e_{h}=(v_j,v_i)\in\mathcal{E}$ represents that $v_i$ can receive the information from $v_{j}$, in which $v_{j}$ is called a neighbor of $v_{i}$. All neighbors of of $v_i$ are collected in $N(i)=\{v_j:(v_{j},v_{i})\in \mathcal{E}\}$. A directed path $\mathcal{P}$ of length $k$ is formed by a finite sequence of edges satisfying $e_{i}=(v_{m_{i-1}},v_{m_{i}})$, $\forall i\in\mathcal{F}_{k}$, where $v_{m_{0}}$, $v_{m_{1}}$, $\cdots$, $v_{m_{k}}$ are different nodes. By contrast, an undirected path $\mathcal{P}_u$ of length $k$ allows $e_{i}=(v_{m_{i-1}},v_{m_{i}})$ or $e_{i}=(v_{m_{i}},v_{m_{i-1}})$, $\forall i\in \mathcal{F}_{k}$. When $\mathcal{P}$ is closed (i.e., $v_{m_{0}}=v_{m_{k}}$), $\mathcal{P}$ is also called a directed cycle $\mathcal{C}$. If $\mathcal{P}_u$ is closed (i.e., any of $v_{m_{0}}=v_{m_{k}}$, $v_{m_{1}}=v_{m_{k}}$, $v_{m_{0}}=v_{m_{k-1}}$ and $v_{m_{1}}=v_{m_{k-1}}$ holds), then $\mathcal{P}_u$ is also called a semi-cycle $\mathcal{C}_{u}$. It is said that a signed digraph $\mathcal{G}$ is strongly connected if there exists a directed path between every distinct pair of nodes. Besides, if the node set $\mathcal{V}$ can be separated into two disjoint subsets $\mathcal{V}_{1}$ and $\mathcal{V}_{2}$ such that $a_{ij}\geq0$, $\forall v_{i},v_{j}\in\mathcal{V}_{1}$, $\forall v_{i},v_{j}\in\mathcal{V}_{2}$ and $a_{ij}\leq0$, $\forall v_{i}\in\mathcal{V}_{1},v_{j}\in \mathcal{V}_{2}$, $\forall v_{j}\in \mathcal{V}_{1},v_{i}\in \mathcal{V}_{2}$, then $\mathcal{G}$ is structurally balanced; and otherwise, $\mathcal{G}$ is structurally unbalanced.

For any signed digraph, a unique unsigned digraph can be correspondingly induced, as shown in the following definition.

{\it Definition 1 (Induced Unsigned Digraph):} Given any signed digraph $\mathcal{G}=\left(\mathcal{V},\mathcal{E},A\right)$, a digraph $\overline{\mathcal{G}}=\left(\mathcal{V},\mathcal{E},\overline{A}\right)$ is called an induced unsigned digraph of $\mathcal{G}$ if its weight matrix $\overline{A}=\left[\overline{a}_{ij}\right]\in\mathbb{R}^{n\times n}$ is defined with entries satisfying $\overline{a}_{ij}=\left|a_{ij}\right|$, $\forall i,j\in\mathcal{F}_{n}$.

Let $\overline{L}$ represent the Lapalcian matrix of $\overline{\mathcal{G}}$. We can validate that $\overline{L}$ satisfies $\overline{L}=\Delta-\overline{A}$. For a particular case when $\mathcal{G}$ is structurally balanced, $\overline{L}=D_{n}LD_{n}$ holds for some gauge transformation $D_n\in \mathcal{D}$. Define $\overline{L}_{ii}\in\mathbb{R}^{(n-1)\times(n-1)}$, $\forall i\in\mathcal{F}_n$ as the matrix that is induced from $\overline{L}$ by removing its $i$th row and $i$th column, with which we correspondingly denote
\begin{equation}\label{equaa3}
W=\mathrm{diag}\left\{\det\left(\overline{L}_{11}\right),\det\left(\overline{L}_{22}\right),\cdots,\det\left(\overline{L}_{nn}\right)\right\}.
\end{equation}

For any signed digraph $\mathcal{G}$ together with its induced unsigned digraph $\overline{\mathcal{G}}$, we can further introduce a unique undirected signed graph $\hat{\mathcal{G}}$ in the following definition.

{\it Definition 2 (Mirror Signed Graph):} For any signed digraph $\mathcal{G}=\left(\mathcal{V},\mathcal{E},A\right)$, an undirected signed graph $\hat{\mathcal{G}}=\left(\mathcal{V},\hat{\mathcal{E}},\hat{A}\right)$ is called a mirror signed graph of $\mathcal{G}$ if its edge set $\hat{\mathcal{E}}$ is associated with a symmetric weight matrix $\hat{A}=\left[\hat{a}_{ij}\right]\in\mathbb{R}^{n\times n}$ given by %whose entries are defined by
\begin{equation}\label{eq:7}
\hat{a}_{ij}%=\hat{a}_{ji}
=\frac{\Ds\det\left(\overline{L}_{ii}\right)a_{ij}+\det\left(\overline{L}_{jj}\right)a_{ji}}{\Ds2},~~~\forall i,j\in\mathcal{F}_{n}.
\end{equation}

For Definition 2, we can validate from (\ref{eq:7}) that the adjacency weight matrix $\hat{A}$ of $\hat{\mathcal{G}}$ can be written as
\begin{equation*}
\hat{A}=\frac{\Ds WA+A^{T}W}{\Ds2}.
\end{equation*}

\noindent Moreover, if the Lapalcian matrix of $\hat{\mathcal{G}}$ is denoted by $\hat{L}$, then a candidate definition of $\hat{L}$ is proposed in the following lemma.

\begin{lemma}\label{lemma4}
For any signed digraph $\mathcal{G}$, the Lapalcian matrix $\hat{L}$ of its mirror signed graph $\hat{\mathcal{G}}$ can be given by
\[
\hat{L}=\frac{\Ds WL+L^{T}W}{\Ds2}.
\]
%
%\noindent where $L$ is the Laplacian matrix of $\mathcal{G}$.
\end{lemma}

\begin{proof}
See Appendix A.
\end{proof}

We proceed to present a useful lemma to reveal the relations between the connectivity and sign pattern properties of $\mathcal{G}$ and of $\hat{\mathcal{G}}$.

\begin{lemma}\label{lemma3}
For a strongly connected signed digraph $\mathcal{G}$ and its mirror signed graph $\hat{\mathcal{G}}$, the following results hold:
\begin{enumerate}
\item[1)]
$\hat{\mathcal{G}}$ is connected together with $\hat{\mathcal{E}}$ given by $\hat{\mathcal{E}}=\mathcal{E}\cup\tilde{\mathcal{E}}$, where $\tilde{\mathcal{E}}=\left\{(v_{i},v_{j}):(v_{j},v_{i})\in \mathcal{E}\right\}$;

\item[2)]
$\hat{\mathcal{G}}$ has the same sign pattern as $\mathcal{G}$ such that $\hat{\mathcal{G}}$ is structurally balanced (respectively, unbalanced) if and only if $\mathcal{G}$ is structurally balanced (respectively, unbalanced).
\end{enumerate}
\end{lemma}

\begin{proof}
See Appendix B.
\end{proof}

As a consequence of Lemma \ref{lemma3}, the following lemma shows the relationship between the Laplacian matrices of $\mathcal{G}$ and of $\hat{\mathcal{G}}$ from the viewpoint of their null spaces.

\begin{lemma}\label{lemma5}
For a strongly connected signed digraph $\mathcal{G}$ and its mirror signed graph $\hat{\mathcal{G}}$, $\mathcal{N}\left(\hat{L}\right)=\mathcal{N}\left(L\right)$ holds, and moreover,
\begin{enumerate}
\item[1)]
$\mathcal{N}\left(\hat{L}\right)=\mathrm{span}\left\{D_{n}1_{n}\right\}$ if and only if $\mathcal{G}$ is structurally balanced, where $D_{n}\in\mathcal{D}$ is such that $\overline{L}=D_{n}LD_{n}$;

\item[2)]
$\mathcal{N}\left(\hat{L}\right)=0_n$ if and only if $\mathcal{G}$ is structurally unbalanced.
\end{enumerate}
\end{lemma}

\begin{proof}
See Appendix C.
\end{proof}

Motivated by Lemma \ref{lemma5}, we can further study the eigenvalue distribution for Laplacian matrices of mirror signed graphs.
\begin{corollary}
For a strongly connected signed digraph $\mathcal{G}$ and its mirror signed graph $\hat{\mathcal{G}}$, let $\lambda_1(\hat{L})$, $\lambda_2(\hat{L})$, $\cdots$, $\lambda_n(\hat{L})$ denote the eigenvalues of $\hat{L}$. The following results hold.
\begin{enumerate}
  \item $0=\lambda_1(\hat{L})<\lambda_2(\hat{L})\leq \cdots \leq \lambda_n(\hat{L})$ if and only if $\hat{\mathcal{G}}$ is structurally balanced.
  \item $0<\lambda_1(\hat{L})\leq \lambda_2(\hat{L}) \leq \cdots \leq \lambda_n(\hat{L})$ if and only if $\hat{\mathcal{G}}$ is structurally unbalanced.
\end{enumerate}
\end{corollary}

\begin{proof}
According to Definition 2, we can realize that $\hat{L}$ is a symmetric matrix and its all eigenvalues are real numbers. Since $\mathcal{G}$ is strongly connected, it follows from Lemma \ref{lemma3} that $\hat{\mathcal{G}}$ is strongly connected. With Lemma \ref{lemma5} and Gersgorin Theorem, we can complete this proof.
\end{proof}

\subsection{Improved Laplacian Potential Function and Its Properties}

The Laplacian potential function can be exploited to measure the total disagreements of all agents, which plays a central role in investigating the distributed control problems of signed networks. On the specific, the distributed control protocol can be induced from the gradient-based feedback of Laplacian potential functions (see e.g., \cite{Olfa04,Alta13}). Besides, the Laplacian potential function is used to develop the convergence analyses for dynamic behaviors of signed networks (see e.g., \cite{Olfa04,Alta13,Jia16,Ying16}), in which the weight balance assumption on signed digraphs is necessary. We can find that the requirement of weight balance may generate limitations on using Laplacian potential functions to work out the distributed control problems of signed networks. To avoid this drawback, we propose a new kind of Laplacian potential functions for signed networks that is named as improved Laplacian potential function and is given in the following definition.

{\it Definition 3 (Improved Laplacian Potential Function):} For any signed digraph $\mathcal{G}$, an improved Laplacian potential function $\Phi_\mathrm{e}(x)$ of the signed network (\ref{eq1}) is defined as
\begin{equation}\label{eq5}
\Phi_\mathrm{e}(x)=\sum_{i=1}^n\sum_{j=1}^{n}\det(\overline{L}_{ii})|a_{ij}|(x_i-\mathrm{sgn}(a_{ij})x_j)^2.
\end{equation}

We provide an expression of $\Phi_\mathrm{e}(x)$ with matrices $L$ and $W$, and give an useful property of $\Phi_\mathrm{e}(x)$ in the following lemma.

\begin{lemma}\label{lem3}
For the signed network (\ref{eq1}) under a signed digraph $\mathcal{G}$, its improved Laplacian potential function $\Phi_\mathrm{e}(x)$ satisfies
\begin{equation}\label{equ:5}
\begin{split}
\Phi_\mathrm{e}(x)=x^T(WL+L^TW)x.
\end{split}
\end{equation}

\noindent Furthermore, when the signed digraph $\mathcal{G}$ is strongly connected, the following results hold.
\begin{enumerate}
  \item $\Phi_\mathrm{e}(x)$ is semi-positive definite and $\Phi_\mathrm{e}(x)=0$ implies $x=cD_n1_n$ for some $c\in \mathbb{R}$ if and only if $\mathcal{G}$ is structurally balanced.
  \item $\Phi_\mathrm{e}(x)$ is positive definite and $\Phi_\mathrm{e}(x)=0$ implies $x=0_n$ if and only if $\mathcal{G}$ is structurally unbalanced.
\end{enumerate}
\end{lemma}

\begin{proof}
See Appendix D.
\end{proof}

\begin{remark}
From \cite{Alta13}, the Laplacian potential function $\Phi(x)$ is given by
\begin{equation*}
\Phi(x)=\sum_{i=1}^n\sum_{j=1}^{n}|a_{ij}|(x_i-\mathrm{sgn}(a_{ij})x_j)^2.
\end{equation*}

\noindent When the signed digraph $\mathcal{G}$ is weight balanced, $\Phi(x)$ satisfies
\begin{equation}\label{equ3}
\Phi(x)=x^T(L+L^T)x.
\end{equation}

\noindent It is worth noticing that the equation (\ref{equ3}) fails once the weight balance condition is broken, which may lead to constrains on the application of $\Phi(x)$ when the associated signed digraph is weight unbalanced (see e.g., \cite{Alta13,Jia16,Ying16}). In contrast to $\Phi(x)$, the improved Laplacian potential function $\Phi_\mathrm{e}(x)$ has a series of coefficients $\det(\overline{L}_{ii})$, $\forall i\in \mathcal{F}_n$. From Lemma \ref{lem3}, we can develop that the equation (\ref{equ:5}) holds regardless of whether the associated signed digraph is weight balanced or not, which makes it possible to employ $\Phi_\mathrm{e}(x)$ to solve distributed control problems of signed networks under both weight balanced and unbalanced signed digraphs. This greatly extends the range of application for Laplacian potential functions in the studies of distributed control problems of signed networks.
\end{remark}

As a consequence of Lemma \ref{lem3}, we can induce the following corollary for improved Laplacian potential functions.

\begin{corollary}\label{cor1}
Consider a strongly connected, signed digraph $\mathcal{G}$. If $\mathcal{G}$ is weight balanced, then $\Phi_\mathrm{e}(x)$ satisfies
\begin{equation}\label{equ:8}
\begin{split}
\Phi_\mathrm{e}(x)&=\alpha\sum_{i=1}^n\sum_{j=1}^{n}|a_{ij}|\left(x_i-\mathrm{sgn}(a_{ij})x_j\right)^2\\
       &=\alpha x^T(L+L^T)x
\end{split}
\end{equation}

\noindent for some $\alpha>0$.
\end{corollary}

\begin{proof}
Since the signed digraph $\mathcal{G}$ is strongly connected and weight balanced, we can directly derive from Definition 1 that $\overline{\mathcal{G}}$ is also strongly connected and weight balanced. It thus follows from \cite{Li09} that there exists some positive constant $\alpha$ such that

\begin{equation*}
\det(\overline{L}_{11})=\det(\overline{L}_{22})=\cdots=\det(\overline{L}_{nn})=\alpha.
\end{equation*}

\noindent Namely, $W=\alpha I_n$ holds due to (\ref{equaa3}). It is immediate to develop (\ref{equ:8}) from (\ref{equ:5}). This proof is complete.
\end{proof}

From (\ref{equ3}) and (\ref{equ:8}), it is obvious that $\Phi(x)$ is only a particular case of $\Phi_\mathrm{e}(x)$ when the signed digraph $\mathcal{G}$ is strongly connected and weight balanced.

\section{Distributed Averaging Protocols and Results}

In this section, we target at proposing two distributed control protocols such that signed networks achieve the signed-average consensus (respectively, stability) if and only if their associated signed digraphs are structurally balanced (respectively, unbalanced), regardless of whether the weight balance condition is satisfied or not. Benefitting from improved Laplacian potential functions, we can further develop the convergence analyses for our two proposed protocols from the perspective of Lyapunov stability theory.

For every $v_i$, we propose the first control protocol based on the nearest neighbor rule as
\begin{equation}\label{equ7}
u_i=-\sum_{v_j\in N(i)}|\hat{a}_{ij}|(x_i-\mathrm{sgn}(\hat{a}_{ij})x_j),~ \forall i\in \mathcal{F}_n.
\end{equation}

\noindent We should point out that the protocol (\ref{equ7}) can also be obtained by the gradient-based feedback of $\Phi_\mathrm{e}(x)$. By using $\hat{L}$, we can write (\ref{eq1}) and (\ref{equ7}) as a compact form of
\begin{equation}\label{equ8}
\dot{x}=-\hat{L}x.
\end{equation}

If the protocol (\ref{equ7}) is applied, then signed-average consensus results can be provided in the following theorem.

\begin{theorem}\label{thm1}
Consider the system (\ref{equ8}) under a signed digraph $\mathcal{G}$ that is strongly connected. Then, the system (\ref{equ8}) can achieve
\begin{enumerate}[1)]
  \item the signed-average consensus if and only if $\mathcal{G}$ is structurally balanced;
  \item the state stability if and only if $\mathcal{G}$ is structurally unbalanced.
\end{enumerate}
\end{theorem}

\begin{proof}
We select a Lyapunov function candidate for the system (\ref{equ8}) as
\begin{equation*}
V(x)=x^Tx.
\end{equation*}

\noindent Taking the derivation of $V$ along (\ref{equ8}) leads to
\begin{equation*}
\begin{split}
\dot{V}(x)&=\dot{x}^Tx+x^T\dot{x}=-x^T(\hat{L}^T+\hat{L})x\\
&=-x^T(WL+L^TW)x=-\Phi_\mathrm{e}(x).
\end{split}
\end{equation*}

``$\Leftarrow$'': 1) When the signed digraph $\mathcal{G}$ is structurally balanced, we can develop $\dot{V}\leq 0$ from Lemma \ref{lem3}. It is immediate to obtain that the trajectories converge to the largest invariant set $S=\{x\in \mathbb{R}^{n}|\dot{V}(x)=0\}$ as $t\rightarrow \infty$ from LaSalle's Theorem \cite[Theorem 4.4]{Khal}. This, together with Lemma \ref{lem3}, guarantees $S=\{x\in \mathbb{R}^{n}|x=c D_n1_n$, $c\in \mathbb{R}\}$, which indicates that the system (\ref{equ8}) can achieve the bipartite consensus.

In the following, we investigate the convergency value of the system (\ref{equ8}). Since $\mathcal{G}$ is strongly connected and structurally balanced, its mirror signed graph $\mathcal{\hat{G}}$ is connected and structurally balanced from Lemma \ref{lemma3}. We can deduce that $\nu_l=D_n1_n$ and $\nu_r=D_n1_n$ are the left and right eigenvectors of $\hat{L}$ associated with zero eigenvalue, respectively. The terminal state of the system (\ref{equ8}) is given by
\begin{equation}\label{equ12}
\begin{split}
x(\infty)&=\lim_{t\rightarrow \infty}e^{-\hat{L}t}x(0)\\
&=\frac{\nu_r\nu_l^T}{\nu_l^T\nu_r}x(0)=\frac{1_n^TD_nx(0)}{n}D_n1_n.
\end{split}
\end{equation}

\noindent From (\ref{equ12}), we realize that the system (\ref{equ8}) can accomplish the signed-average consensus objective when the signed digraph $\mathcal{G}$ is structurally balanced.

2) When the signed digraph $\mathcal{G}$ is structurally unbalanced, it follows from Lemma \ref{lemma3} that $\mathcal{\hat{G}}$ is structurally unbalanced and $\hat{L}$ is a positive definite matrix, which leads to $\dot{V}(x)<0$. This, together with \cite[Theorem 4.1]{Khal}, can guarantee that the system (\ref{equ8}) is asymptotically stable.

``$\Rightarrow$'': 1) We employ the proof by contradiction and assume that $\mathcal{G}$ is structurally unbalanced. It follows from Lemma \ref{lemma3} that $\mathcal{\hat{G}}$ is also structurally unbalanced and $-\hat{L}$ is Hurwitz stable. Hence, the system (\ref{equ8}) is asymptotically stable, which causes a contradiction on the system (\ref{equ8}) achieving the signed-average consensus. On the contrary, $\mathcal{G}$ is structurally balanced.

2) Suppose that $\mathcal{G}$ is structurally balanced. The system (\ref{equ8}) can reach the bipartite consensus, which causes a contradiction on the state stability of the system (\ref{equ8}). Thus, $\mathcal{G}$ is structurally unbalanced. We complete this proof.
\end{proof}

\begin{remark}
According to theorem \ref{thm1}, we know that the signed-average consensus problems of signed networks under strongly connected signed digraphs can be equivalently converted into the bipartite consensus problems of signed networks under the associated mirror signed graphs, which gives a new perspective to explore distributed averaging problems of signed networks.
\end{remark}

Different from the protocol (\ref{equ7}), we give the second protocol based on the nearest neighbor rule as
\begin{equation}\label{eq10}
u_i=-\sum_{v_j\in N(i)}\det(\overline{L}_{ii})|a_{ij}|(x_i-\mathrm{sgn}(a_{ij})x_j),~i\in \mathcal{F}_n.
\end{equation}

\begin{remark}
In comparison with the classical distributed protocol first proposed in \cite{Alta13}
\begin{equation}\label{equ11}
u_i=-\sum_{v_j\in N(i)}|a_{ij}|(x_i-\mathrm{sgn}(a_{ij})x_j),~i\in \mathcal{F}_n,
\end{equation}

\noindent the protocol (\ref{eq10}) consists of a control gain $\det(\overline{L}_{ii})$, $i\in \mathcal{F}_n$. When $\mathcal{G}$ is strongly connected and weight balanced, with the proof of Corollary \ref{cor1}, the protocol (\ref{eq10}) becomes
\begin{equation}\label{eq11}
u_i=-\alpha\sum_{v_j\in N(i)}|a_{ij}|[x_i-\mathrm{sgn}(a_{ij})x_j],~i\in \mathcal{F}_n.
\end{equation}

\noindent Compared with the protocols (\ref{equ11}) and (\ref{eq11}), we can easily find that the protocol (\ref{equ11}) is a particular case of the protocol (\ref{eq11}), in which the gain $\alpha$ just has an effect on the convergence rate.
\end{remark}

Applying the protocol (\ref{eq10}) to the signed network (\ref{eq1}) yields
\begin{equation}\label{eq9}
\dot{x}=-WLx.
\end{equation}

The following lemma discloses the distribution of eigenvalues for the matrix $WL$.

\begin{lemma}\label{lem5}
Consider a strongly connected signed digraph $\mathcal{G}$. The following results for $WL$ hold.
\begin{enumerate}
  \item $WL$ has a zero eigenvalue and all other eigenvalues with positive real parts if and only if $\mathcal{G}$ is structurally balanced.
  \item All eigenvalues of $WL$ have positive real parts if and only if $\mathcal{G}$ is structurally unbalanced.
\end{enumerate}
\end{lemma}

\begin{proof}
Because the matrix $W$ is a diagonal matrix and all diagonal elements are positive real numbers. We immediately obtain these results from \cite[Theorems 4.1 and 4.2]{Meng18}.
\end{proof}

With Lemma \ref{lem5}, we are in position to present signed-average consensus results of signed networks in the following theorem.

\begin{theorem}\label{thm2}
Consider the system (\ref{eq9}) under a signed digraph $\mathcal{G}$ that is strongly connected. Then, the system (\ref{eq9}) can reach
\begin{enumerate}[1)]
  \item the signed-average consensus if and only if $\mathcal{G}$ is structurally balanced;
  \item the state stability if and only if $\mathcal{G}$ is structurally unbalanced.
\end{enumerate}
\end{theorem}

\begin{proof}
We construct a Lyapunov function candidate
\begin{equation*}
V(x)=x^Tx.
\end{equation*}

\noindent Taking the derivation of $V$ along (\ref{eq9}) yields
\begin{equation*}
\begin{split}
\dot{V}(x)=\dot{x}^Tx+x^T\dot{x}=-x^T(L^TW+WL)x=-\Phi_\mathrm{e}(x).
\end{split}
\end{equation*}

1) When the signed digraph $\mathcal{G}$ is structurally balanced, there exists a matrix $D_n\in \mathcal{D}$ such that $\overline{L}=D_nLD_n$ holds. From \cite{Li09}, we can realize $[\det(\overline{L}_{11})$, $\det(\overline{L}_{22})$, $\cdots$, $\det(\overline{L}_{nn})]\overline{L}$ $=$ $0_n^T$ that indicates $\det(\overline{L}_{ii})\sum_{j=1}^n|a_{ij}|=\sum_{j=1}^n\det(\overline{L}_{jj})|a_{ji}|$, $\forall i\in \mathcal{F}_n$. It directly produces $1_n^TW\overline{L}=0_n^T$. By $\overline{L}=D_nLD_n$, we have $1_n^TW\overline{L}=1_n^TWD_nLD_n=1_n^TD_nWLD_n=0_n^T$. Thus, $\nu_l=D_n1_n$ and $\nu_r=D_n1_n$ are the left and right eigenvectors of $WL$ corresponding to zero eigenvalue, respectively. The rest of proof is similar to the proof of Theorem \ref{thm1}.

2) When $\mathcal{G}$ is structurally unbalanced, the proof is same as the proof of Theorem \ref{thm1} and we omit it for simplicity.
\end{proof}

\begin{remark}
From Theorems \ref{thm1} and \ref{thm2}, we can solve the signed-average consensus issues for any directed signed networks, in which two distributed control protocols are proposed to ensure the signed-average consensus no matter whether the associated signed digraphs are weight balanced or not. It greatly extends the existing bipartite consensus results of signed networks (see e.g., \cite{Alta13}). Because unsigned networks are a particular case of signed networks, our proposed protocols (\ref{equ7}) and (\ref{eq10}) can also solve average consensus issues of unsigned networks although their associated unsigned digraphs are weight unbalanced.
\end{remark}

\begin{remark}
From \cite[Remark 2]{Alta13}, it follows that the Laplacian potential function $\Phi(x)$ can be exploited to derive the convergence analysis of the control protocol (\ref{equ11}) when the associated signed digraph is strongly connected and weight balanced. It is worth noticing that $\Phi(x)$ does not work if the weight balance condition is broken. According to the proofs of Theorems \ref{thm1} and \ref{thm2}, we find that this drawback can be removed by employing the improved Laplacian potential function $\Phi_\mathrm{e}(x)$ that gives an alternative approach to constructing Lyapunov functions for convergence analysis of signed networks regardless of whether the weight balance condition is satisfied or not.
\end{remark}

\section{An Application of Improved Laplacian Potential Functions}

The fixed-time consensus problems have been explored for signed networks subject to undirected topologies in \cite{Jia16}, from which nonlinear distributed protocols are proposed to ensure all agents achieving the bipartite consensus (or state stability) within a fixed time and the corresponding convergence analyses can be derived by taking advantage of Laplacian potential functions. However, the existing results in \cite{Jia16} are ineffective for signed networks under directed topologies. In this section, benefitting from improved Laplacian potential functions, we concentrate on studying how to extend the existing fixed-time consensus results to directed signed networks.

For any initial states $x_{i0}$, $\forall i\in \mathcal{F}_n$, we say that the signed network (\ref{eq1}) achieves
\begin{enumerate}
\item fixed-time bipartite consensus if there exist some scalars $\sigma_{i}\in\{\pm1\}$ and $c\neq0$ such that
\begin{equation}\label{equa15}
\left\{
\begin{aligned}
&\lim_{t\rightarrow T}x_i(t)=\sigma_ic\\
&x_i(t)=\sigma_ic,~\forall t\geq T
\end{aligned},\quad i\in \mathcal{F}_n%~\mbox{and}~c\in \mathbb{R}
\right.
\end{equation}
\item fixed-time state stability if
\begin{equation}\label{equa16}
\left\{
\begin{aligned}
&\lim_{t\rightarrow T}x_i(t)=0\\
&x_i(t)=0,~\forall t\geq T
\end{aligned},\quad i\in \mathcal{F}_n
\right.
\end{equation}
\end{enumerate}

\noindent where $T\in [0,\infty)$ is the setting time that is independent of the initial states $x_{i0}$, $\forall i\in \mathcal{F}_n$. In particular, when $c$ satisfies $c=\frac{\Ds 1}{\Ds n}\sum_{j=1}^{n}\sigma_{j}x_{j0}$, we say that the signed network (\ref{eq1}) achieves fixed-time signed-average consensus.

In order to achieve the fixed-time consensus objective (\ref{equa15}) and (\ref{equa16}), we propose a distributed control protocol as follows
\begin{equation}\label{equaa16}
\begin{split}
u_i&=k_1\left(-\sum_{v_j\in N(i)}|\hat{a}_{ij}|(x_i-\mathrm{sgn}(\hat{a}_{ij})x_j)\right)^{\frac{m}{r}}\\
&~~~+k_2\left(-\sum_{v_j\in N(i)}|\hat{a}_{ij}|(x_i-\mathrm{sgn}(\hat{a}_{ij})x_j)\right)^{\frac{p}{q}},\quad i\in \mathcal{F}_n
\end{split}
\end{equation}

\noindent where $k_1>0$ and $k_2>0$ are control gains, and $m$, $r$, $p$, and $q$ are positive odd integers satisfying $m>r$ and $q>p$.

With the protocol (\ref{equaa16}) being employed, we can develop the fixed-time consensus results in the following theorem.

\begin{theorem}\label{thm3}
For the signed network (\ref{eq1}) whose communication topology is described by a strongly connected signed digraph $\mathcal{G}$, let the protocol (\ref{equaa16}) be applied. If $\mathcal{G}$ is structurally balanced, then the signed network (\ref{eq1}) can achieve the fixed-time bipartite consensus and the setting time $T$ satisfies
\begin{equation}\label{equ:26}
T\leq \frac{1}{\lambda_2(\hat{L})}\left(\frac{n^{\frac{m-r}{2r}}}{k_1}\frac{r}{m-r}+\frac{1}{k_2}\frac{q}{q-p}\right).
\end{equation}

If $\mathcal{G}$ is structurally unbalanced, then the signed network (\ref{eq1}) can accomplish the fixed-time state stability objective and the setting time $T$ satisfies
\begin{equation}\label{equ:27}
T\leq \frac{1}{\lambda_1(\hat{L})}\left(\frac{n^{\frac{m-r}{2r}}}{k_1}\frac{r}{m-r}+\frac{1}{k_2}\frac{q}{q-p}\right).
\end{equation}
\end{theorem}

To exhibit this proof, we need the following two lemmas.

\begin{lemma}{\cite{Jia16}}\label{lem6}
Consider a series of nonnegative real numbers $\eta_1$, $\eta_2$, $\cdots$, $\eta_n$. The following two results hold.
\begin{enumerate}
\item If $\varepsilon>1$, then
\begin{equation*}
\sum_{i=1}^n\eta_i^\varepsilon\geq n^{1-\varepsilon}\left(\sum_{i=1}^n\eta_i\right)^\varepsilon.
\end{equation*}
\item If $0<\varepsilon\leq 1$, then
\begin{equation*}
\sum_{i=1}^n\eta_i^\varepsilon\geq \left( \sum_{i=1}^n\eta_i\right)^{\varepsilon}.
\end{equation*}
\end{enumerate}
\end{lemma}

\begin{lemma}{\cite{Jia16}}
Consider a nonlinear system
\begin{equation*}
\dot{y}=-\alpha_1y^{\frac{m}{r}}-\alpha_2y^{\frac{p}{q}}
\label{lem11}
\end{equation*}

\noindent where $m$, $r$, $p$ and $q$ are positive odd integers satisfying $m>r$ and $q>p$. If $\alpha_1>0$ and $\alpha_2>0$, then for any initial state $y(0)$, $\lim_{t\rightarrow T}y(t)=0$ and $y(t)=0$, $\forall t\geq T$, where the setting time $T$ is given by
\begin{equation*}
T\leq \frac{1}{\alpha_1}\frac{r}{m-r}+\frac{1}{\alpha_2}\frac{q}{q-p}.
\end{equation*}
\end{lemma}

\begin{proof}[Proof of Theorem \ref{thm3}]
We directly resort to the improved Laplacian potential function $\Phi_{\mathrm{e}}(x)$. Then the derivative of $\Phi_{\mathrm{e}}(x)$ along the trajectory of (\ref{eq1}) is given by
\begin{equation*}
\begin{split}
\dot{\Phi}_\mathrm{e}(x)=4\sum_{i=1}^{n}\left(\sum_{j=1}^{n}|\hat{a}_{ij}|(x_i-\mathrm{sgn}(\hat{a}_{ij})x_j)\right)u_i.
\end{split}
\end{equation*}

\noindent By using (\ref{equaa16}), one has
\begin{equation}\label{eq:24}
\begin{split}
\dot{\Phi}_\mathrm{e}(x)&=-4k_1\sum_{i=1}^{n}\left(\left(\sum_{j=1}^{n}|\hat{a}_{ij}|(x_i-\mathrm{sgn}(\hat{a}_{ij})x_j)\right)^2\right)^{\frac{m+r}{2r}}\\
&~~~-4k_2\sum_{i=1}^{n}\left(\left(\sum_{j=1}^{n}|\hat{a}_{ij}|(x_i-\mathrm{sgn}(\hat{a}_{ij})x_j)\right)^2\right)^{\frac{p+q}{2q}}.
\end{split}
\end{equation}

\noindent Due to $(m+r)/(2r)>1$, it follows from Lemma \ref{lem6} that
\begin{equation*}
\begin{split}
&\sum_{i=1}^{n}\left(\left(\sum_{j=1}^{n}|\hat{a}_{ij}|(x_i-\mathrm{sgn}(\hat{a}_{ij})x_j)\right)^2\right)^{\frac{m+r}{2r}}\\
&~~~~~~\geq n^{\frac{r-m}{2r}}\left(\sum_{i=1}^{n}\left(\sum_{j=1}^{n}|\hat{a}_{ij}|(x_i-\mathrm{sgn}(\hat{a}_{ij})x_j)\right)^2\right)^{\frac{m+r}{2r}}.
\end{split}
\end{equation*}

\noindent With $0<(p+q)/(2q)<1$, it is immediate to obtain
\begin{equation*}
\begin{split}
&\sum_{i=1}^{n}\left(\left(\sum_{j=1}^{n}|\hat{a}_{ij}|(x_i-\mathrm{sgn}(\hat{a}_{ij})x_j)\right)^2\right)^{\frac{p+q}{2q}}\\
&~~~~~~~~\geq \left(\sum_{i=1}^{n}\left(\sum_{j=1}^{n}|\hat{a}_{ij}|(x_i-\mathrm{sgn}(\hat{a}_{ij})x_j)\right)^2\right)^{\frac{p+q}{2q}}
\end{split}
\end{equation*}

\noindent from Lemma \ref{lem6}. Hence, the equation (\ref{eq:24}) can be written as
\begin{equation}\label{eq25}
\begin{split}
\dot{\Phi}_\mathrm{e}(x)
\leq -4k_1n^{\frac{r-m}{2r}}(x^T\hat{L}^T\hat{L}x)^{\frac{r+m}{2r}}-4k_2(x^T\hat{L}^T\hat{L}x)^{\frac{p+q}{2q}}.
\end{split}
\end{equation}

When the signed digraph $\mathcal{G}$ is structurally balanced, using \cite[Lemma 4]{Jia16} to (\ref{eq25}) can derive
\begin{equation}\label{eq29}
\begin{split}
\dot{\Phi}_\mathrm{e}(x)\leq &-4k_1n^{\frac{r-m}{2r}}\Big(\frac{\lambda_2(\hat{L})\Phi_\mathrm{e}(x)}{2}\Big)^{\frac{r+m}{2r}}\\
&~~~~~~~~~~~~~~~~~~~~~~~-4k_2\Big(\frac{\lambda_2(\hat{L})\Phi_\mathrm{e}(x)}{2}\Big)^{\frac{p+q}{2q}}.
\end{split}
\end{equation}

\noindent Define $\Psi(x)=\sqrt{\frac{\lambda_2(\hat{L})\Phi_\mathrm{e}(x)}{2}}$ and we can calculate
\begin{equation*}
\dot{\Psi}(x)=\frac{\lambda_2(\hat{L})\dot{\Phi}_\mathrm{e}(x)}{2\sqrt{2\lambda_2(\hat{L})\Phi_\mathrm{e}(x)}}.
\end{equation*}

\noindent Then, (\ref{eq29}) can be rewritten as
\begin{equation}\label{eq30}
\dot{\Psi}(x)\leq-k_1n^{\frac{r-m}{2r}}\lambda_2(\hat{L})\Psi(x)^{\frac{m}{r}}-k_2\lambda_2(\hat{L})\Psi(x)^{\frac{p}{q}}.
\end{equation}

\noindent Denote $z$ as the solution of the following equation
\begin{equation*}
\dot{z}=-k_1n^{\frac{r-m}{2r}}\lambda_2(\hat{L})z^{\frac{m}{r}}-k_2\lambda_2(\hat{L})z^{\frac{p}{q}}.
\end{equation*}

\noindent Based on Lemma \ref{lem11}, we have
\begin{equation*}
\lim_{t\rightarrow T}z(t)=0~\mbox{and}~z(t)=0,~\forall t\geq T
\end{equation*}

\noindent where the setting time $T$ satisfies (\ref{equ:26}). Using the comparison principle \cite[Lemma 3.4]{Khal} to (\ref{eq30}), we can develop $\Psi(x)\leq z$ that implies
\begin{equation*}
\lim_{t\rightarrow T}\Phi_\mathrm{e}(x)=0~\mbox{and}~\Phi_\mathrm{e}(x)=0,~\forall t\geq T
\end{equation*}

\noindent where $T$ is provided by (\ref{equ:26}). It hence follows from Lemma \ref{lem3} that the fixed-time bipartite consensus can be achieved.

When the signed digraph $\mathcal{G}$ is structurally unbalanced, we apply $x^T\hat{L}^T\hat{L}x\geq\lambda_1(\hat{L})V_1(x)$ to (\ref{eq25}) and can derive
\begin{equation*}
\begin{split}
\dot{\Phi}_\mathrm{e}(x)\leq &-4k_1n^{\frac{r-m}{2r}}\Big(\frac{\lambda_1(\hat{L})\Phi_\mathrm{e}(x)}{2}\Big)^{\frac{r+m}{2r}}\\
&~~~~~~~~~~~~~~~~~~~~~~-4k_2\Big(\frac{\lambda_1(\hat{L})\Phi_\mathrm{e}(x)}{2}\Big)^{\frac{p+q}{2q}}.
\end{split}
\end{equation*}

\noindent The rest of proof is similar to the proof of structurally balanced case, and thus its proof details are omitted for simplicity. The proof of Theorem \ref{thm3} is complete.
\end{proof}

\begin{remark}
Through Theorem \ref{thm3}, we provide an application of improved Laplacian potential functions, from which the fixed-time consensus issues can be solved for signed networks with directed topologies. This greatly enhances the existing fixed-time consensus results (see, e.g., \cite{Jia16}) that only consider signed networks under undirected communication topologies. In comparison with \cite{Zhao17,Wang18,Liu19,Lu20}, Theorem \ref{thm3} presents a new way to design fixed-time convergence protocols of directed signed networks, where the setting time in our established results has no relationships with the initial states of agents.
\end{remark}

\begin{remark}
It is worth noticing that the improved Laplacian potential function $\Phi_\mathrm{e}(x)$ can also be employed to solve finite-time consensus problems of directed signed networks. To be specific, let us present a distributed protocol given by
\begin{equation}\label{equa25}
\begin{split}
u_i&=-\mathrm{sgn}\left(\sum_{v_j\in N(i)}|\hat{a}_{ij}|(x_i-\mathrm{sgn}(\hat{a}_{ij})x_j)\right)\\
&~~~~\times \left|\sum_{v_j\in N(i)}|\hat{a}_{ij}|(x_i-\mathrm{sgn}(\hat{a}_{ij})x_j)\right|^\alpha,~~\forall i\in \mathcal{F}_n
\end{split}
\end{equation}

\noindent where $0<\alpha<1$. With the help of $\Phi_\mathrm{e}(x)$, we can develop that the distributed protocol (\ref{equa25}) guarantees the signed network (\ref{eq1}) to achieve the bipartite consensus (respectively, state stability) when the corresponding signed digraph is structurally balanced (respectively, unbalanced) within a finite time $T$. However, for (\ref{equa25}), the finite time $T$ relies on the initial states of agents. Like \cite{Jia16}, the corresponding proof of this finite-time convergence result can be obtained in the same way as that of Theorem \ref{thm3}.
\end{remark}

\section{Simulations}

In this section, we introduce two examples to illustrate the developed theoretical results. We employ two signed digraphs in Fig. \ref{tp1} to describe the interactions among agents. It is obvious from Fig. \ref{tp1} that $\mathcal{G}_a$ and $\mathcal{G}_b$ are both strongly connected and weight unbalanced.

\begin{figure}[H]
  \centering
  \includegraphics[width=8.5cm,clip]{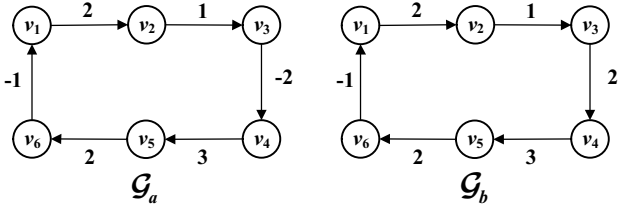}
  \caption{\small{Two signed digraphs $\mathcal{G}_a$ and $\mathcal{G}_b$. Left: $\mathcal{G}_a$ is structurally balanced. Right: $\mathcal{G}_b$ is structurally unbalanced.}}	
  \label{tp1}
\end{figure}

\begin{example}
The initial states of agents are provided by
\begin{equation*}
x(0)=\left[1,2,3,4,5,6\right]^T.
\end{equation*}

\noindent Consider the signed network (\ref{eq1}) under the signed digraph $\mathcal{G}_a$ in Fig. \ref{tp1}. Because $\mathcal{G}_a$ is structurally balanced, there exists a gauge transformation $D_6$ $=$ $\mathrm{diag}\{1$, $1$, $1$, $-1$, $-1$, $-1\}$. The signed network (\ref{eq1}) can achieve the signed-average consensus if the terminal states of agents satisfy
\begin{equation*}
\lim_{t\rightarrow \infty}x_i(t)\in\{1.5,-1.5\},~~~\forall i\in\mathcal{F}_{6}.
\end{equation*}

\begin{figure}[!htbp]
\centering
\includegraphics[width=8cm]{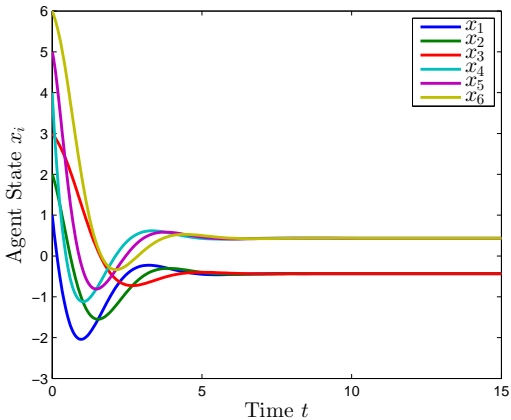}
\caption{\small Bipartite consensus of the signed network (\ref{eq1}) by employing the protocol (\ref{equ11}).}
\label{simi2}
\end{figure}

\begin{figure}[!htbp]
\centering
\includegraphics[width=8cm,clip]{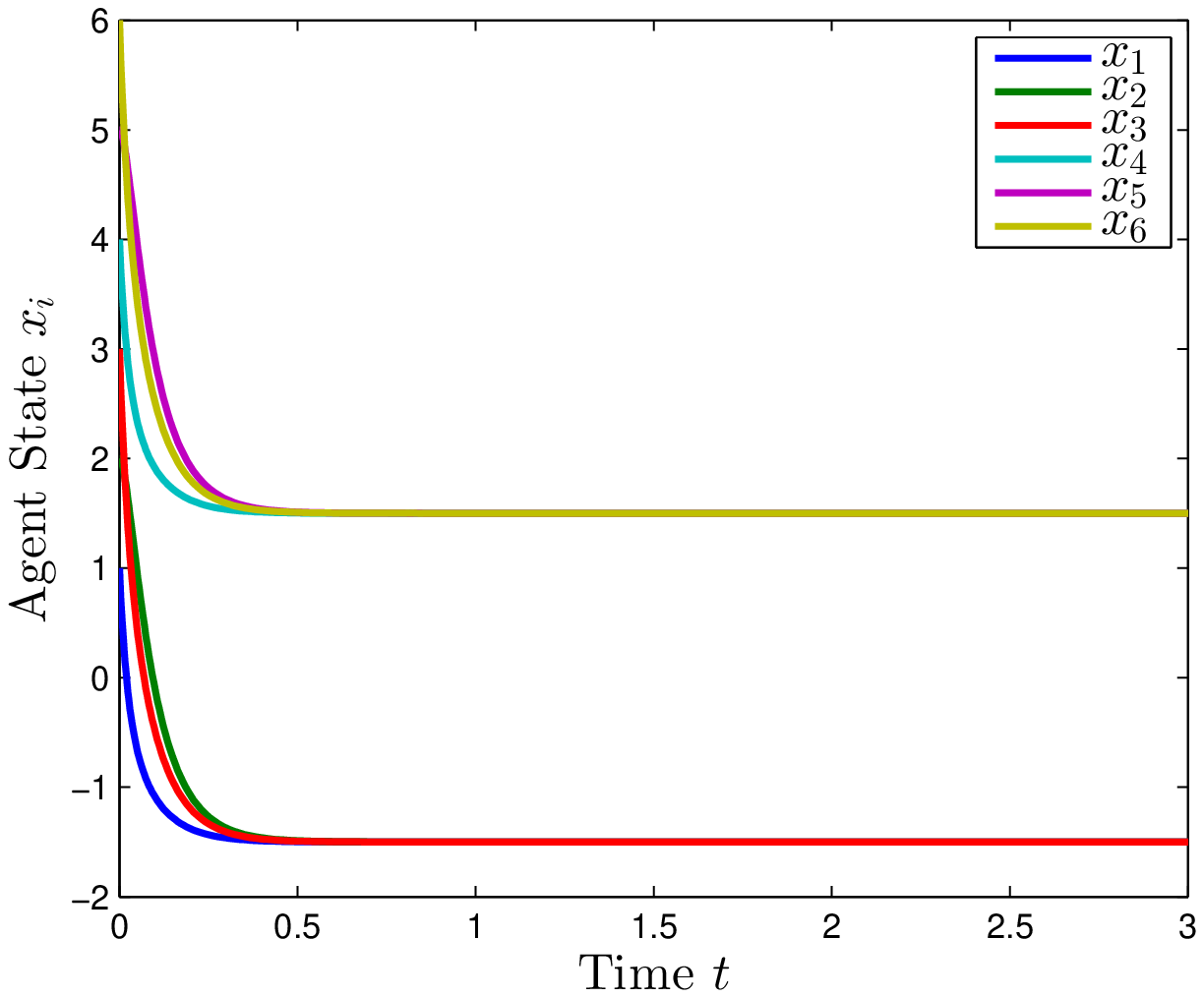}	
\centering
\includegraphics[width=8cm,clip]{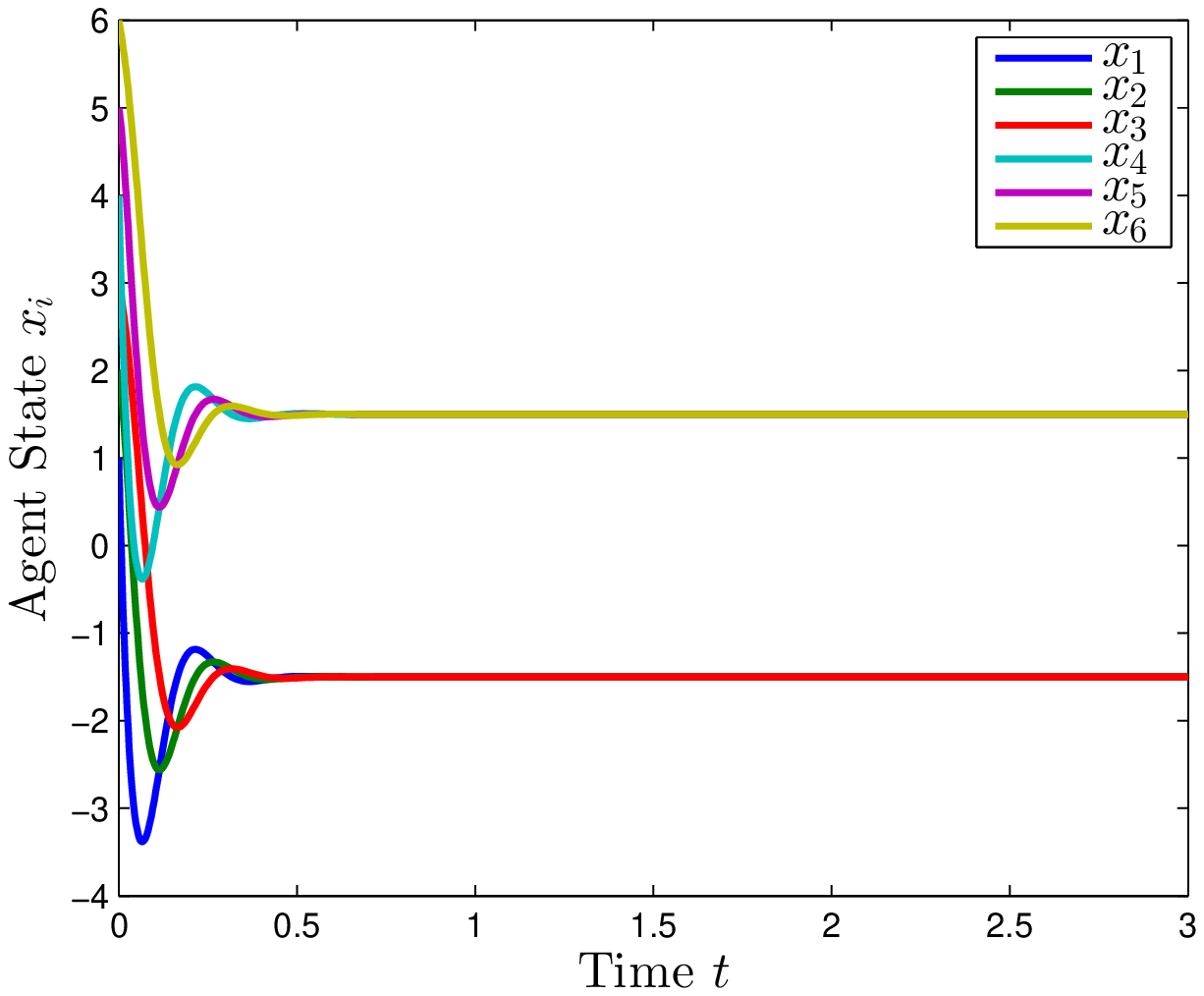}
\caption{\small Signed-average consensus of the signed network (\ref{eq1}). Upper: Under the protocol (\ref{equ7}). Lower: Under the protocol (\ref{eq10}).}	
\label{sim1}	
\end{figure}

By using the protocol (\ref{equ11}), we can plot the state evolution of the signed network (\ref{eq1}) in Fig. \ref{simi2}. This figure clearly depicts that the state $x_i$, $\forall i\in \mathcal{F}_6$ can reach the bipartite consensus with modulus value $0.4348$. Thus, the signed-average consensus can not be achieved with (\ref{equ11}). By applying the protocols (\ref{equ7}) and (\ref{eq10}) to the signed network (\ref{eq1}), the state evolutions of all agents are plotted in Fig. \ref{sim1}. From Fig. \ref{sim1}, we realize that the states of agents polarize with the polarized values $1.5$ and $-1.5$, which are consistent with the signed-average consensus results based on Theorems \ref{thm1} and \ref{thm2}.

Let the signed digraph $\mathcal{G}_b$ be the communication topology of the signed network (\ref{eq1}). Different from $\mathcal{G}_a$, $\mathcal{G}_b$ is structurally unbalanced. When the distributed protocols (\ref{equ11}), (\ref{equ7}) and (\ref{eq10}) are applied to the signed network (\ref{eq1}), the state evolutions of all agents can be plotted in Figs. \ref{simi4} and \ref{simi5}, respectively. We can easily see from Figs. \ref{simi4} and \ref{simi5} that all agents converge to zero. The simulation tests coincide with the state stability results of Theorems \ref{thm1} and \ref{thm2}.

\begin{figure}[!htbp]
\centering
\includegraphics[width=8cm]{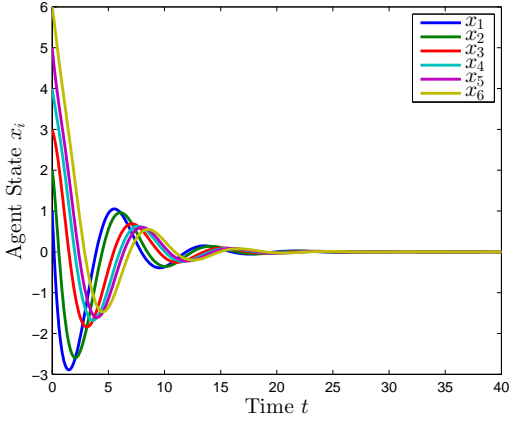}
\caption{\small State stability of the signed network (\ref{eq1}) by employing the protocol (\ref{equ11}).}.
\label{simi4}
\end{figure}

\begin{figure}[!htbp]
\centering
\includegraphics[width=8cm,clip]{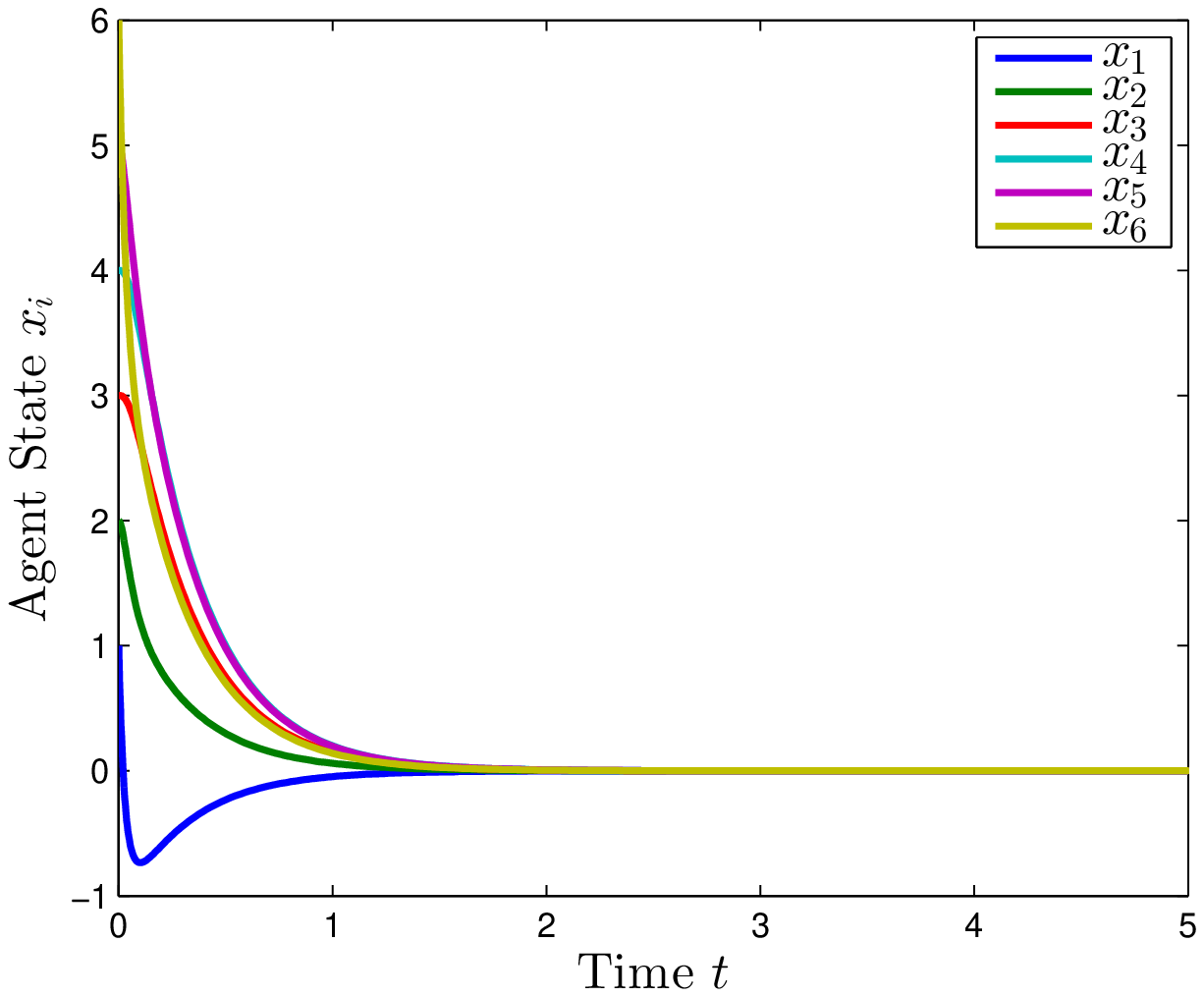}	
\centering
\includegraphics[width=8cm,clip]{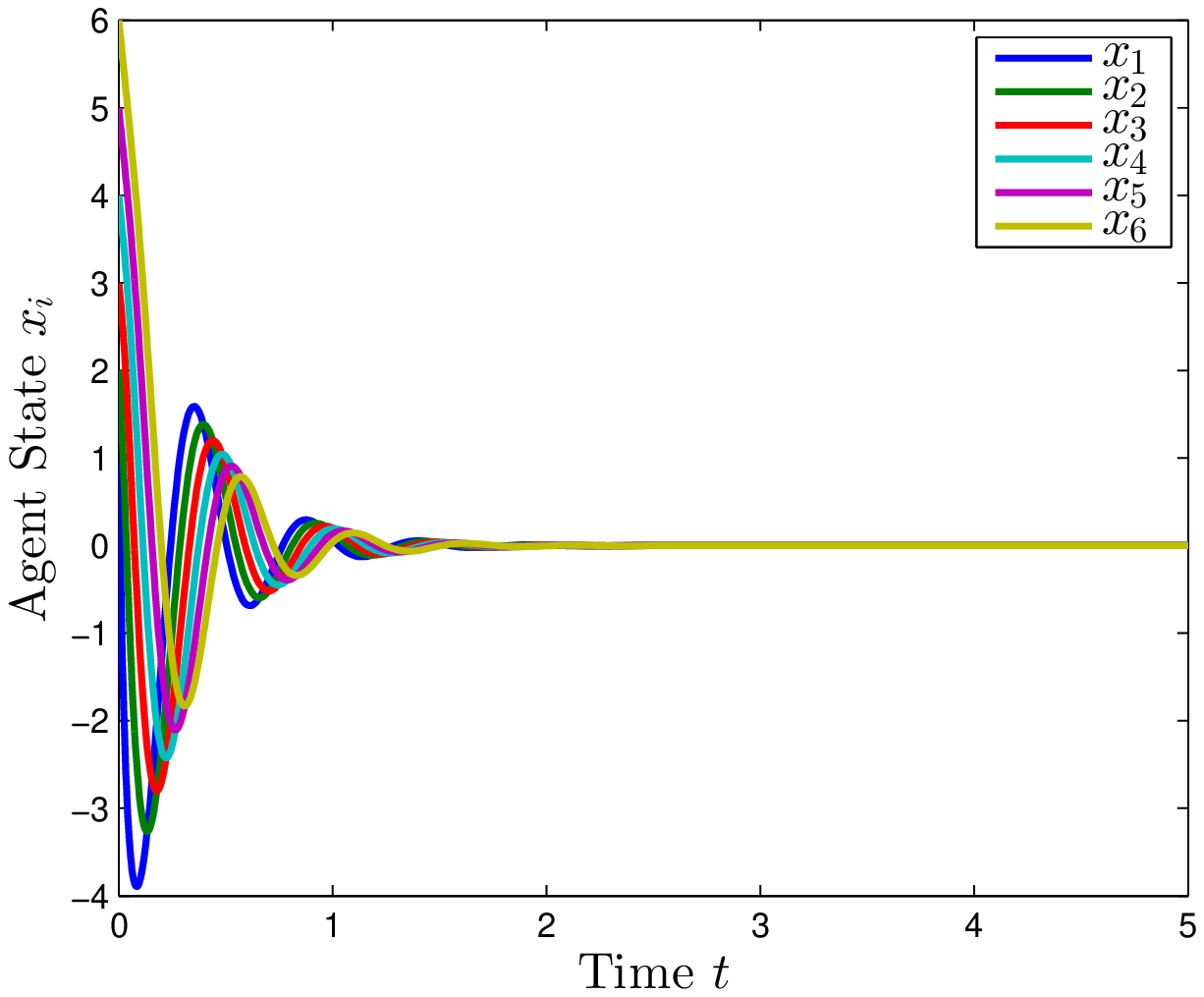}
\caption{\small State stability of the signed network (\ref{eq1}). Upper: Under the protocol (\ref{equ7}). Lower: Under the protocol (\ref{eq10}).}	
\label{simi5}	
\end{figure}
\end{example}

\begin{example}
We use the signed digraphs $\mathcal{G}_a$ and $\mathcal{G}_b$ to denote the communication topology of the signed network (\ref{eq1}). The initial states of agents are provided by
\begin{enumerate}
  \item $x(0)=[2,2,3,4,3,6]^T$
  \item $x(0)=[5,-10,15,-5,10,-15]^T$.
\end{enumerate}

\noindent The nonlinear distributed protocol (\ref{equaa16}) is applied with $k_1=k_2=1$, $m=9$, $r=7$, $p=3$ and $q=5$. The estimation of the setting time $T$ can be obtained from (\ref{equ:26}) and (\ref{equ:27}) (see Table \Rmnum{1} for more details).

\begin{table}[!htbp]
\caption{The setting time $T$ with different initial states}
\centering
\begin{tabular}{|c|c|c|}
\hline
$T$($\leq$)&$[2,2,3,4,3,6]^T$&$[5,-10,15,-5,10,-15]^T$\\
\hline
$\mathcal{G}_a$&$0.5850$&$0.5850$\\
\hline
$\mathcal{G}_b$&$2.1835$&$2.1835$\\
\hline
\end{tabular}
\end{table}

We can see from Table \Rmnum{1} that the setting time $T$ does not rely on the initial state of signed network (\ref{eq1}) regardless of whether the associated signed digraph is structurally balanced or not. To verify this observation, the dynamic behaviors of the signed network (\ref{eq1}) employing the protocol (\ref{equaa16}) are exhibited in Figs. \ref{sim4} and \ref{sim5}. This two figures clearly present that when the protocol (\ref{equaa16}) is applied, the signed network (\ref{eq1}) can achieve bipartite consensus under the signed digraph $\mathcal{G}_a$ and state stability under the signed digraph $\mathcal{G}_b$ within a fixed setting time $T$ that has no relationship with initial states of agents. The simulation results are in accord with Theorem \ref{thm3}.

\begin{figure}[!htbp]
\centering
\includegraphics[width=8cm]{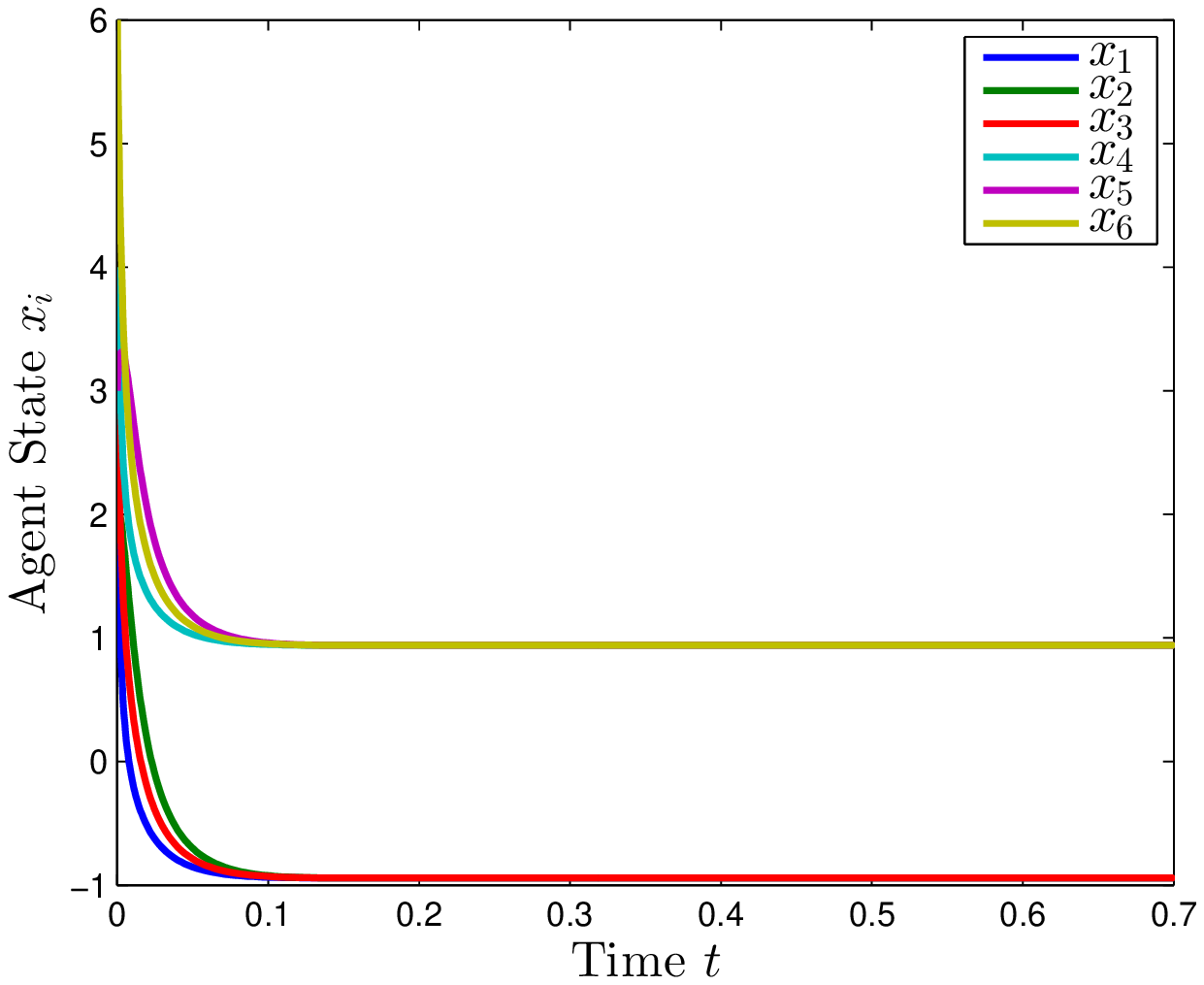}
\centering
\includegraphics[width=8cm]{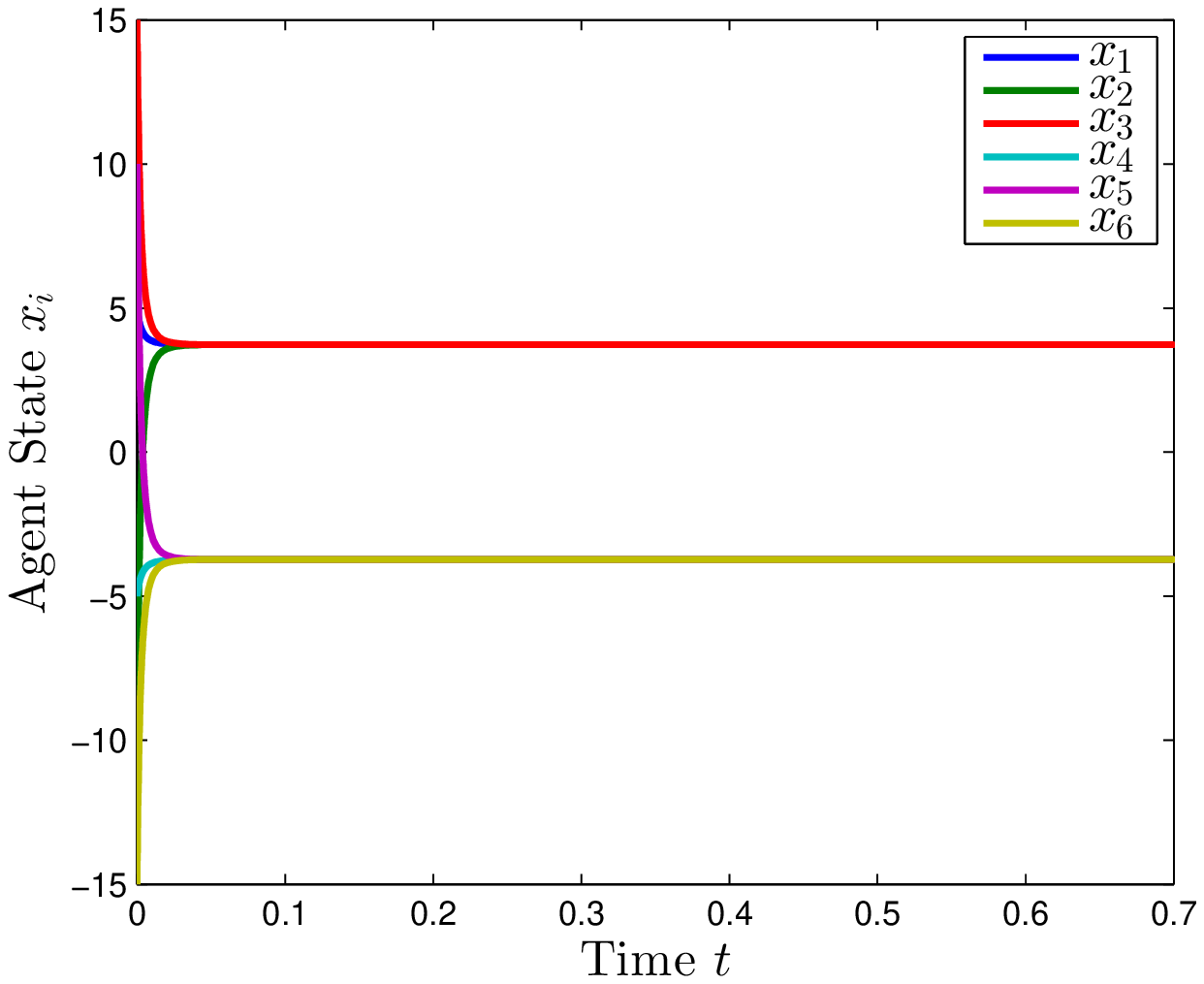}
\caption{\small Fixed-time bipartite consensus under the structurally balanced signed digraph $\mathcal{G}_a$. Upper: $x(0)=[2$, $2$, $3$, $4$, $3$, $6]^T$. Lower: $x(0)=[5$, $-10$, $15$, $-5$, $10$, $-15]^T$.}
\label{sim4}
\end{figure}

\begin{figure}[!htbp]
\centering
\includegraphics[width=8cm]{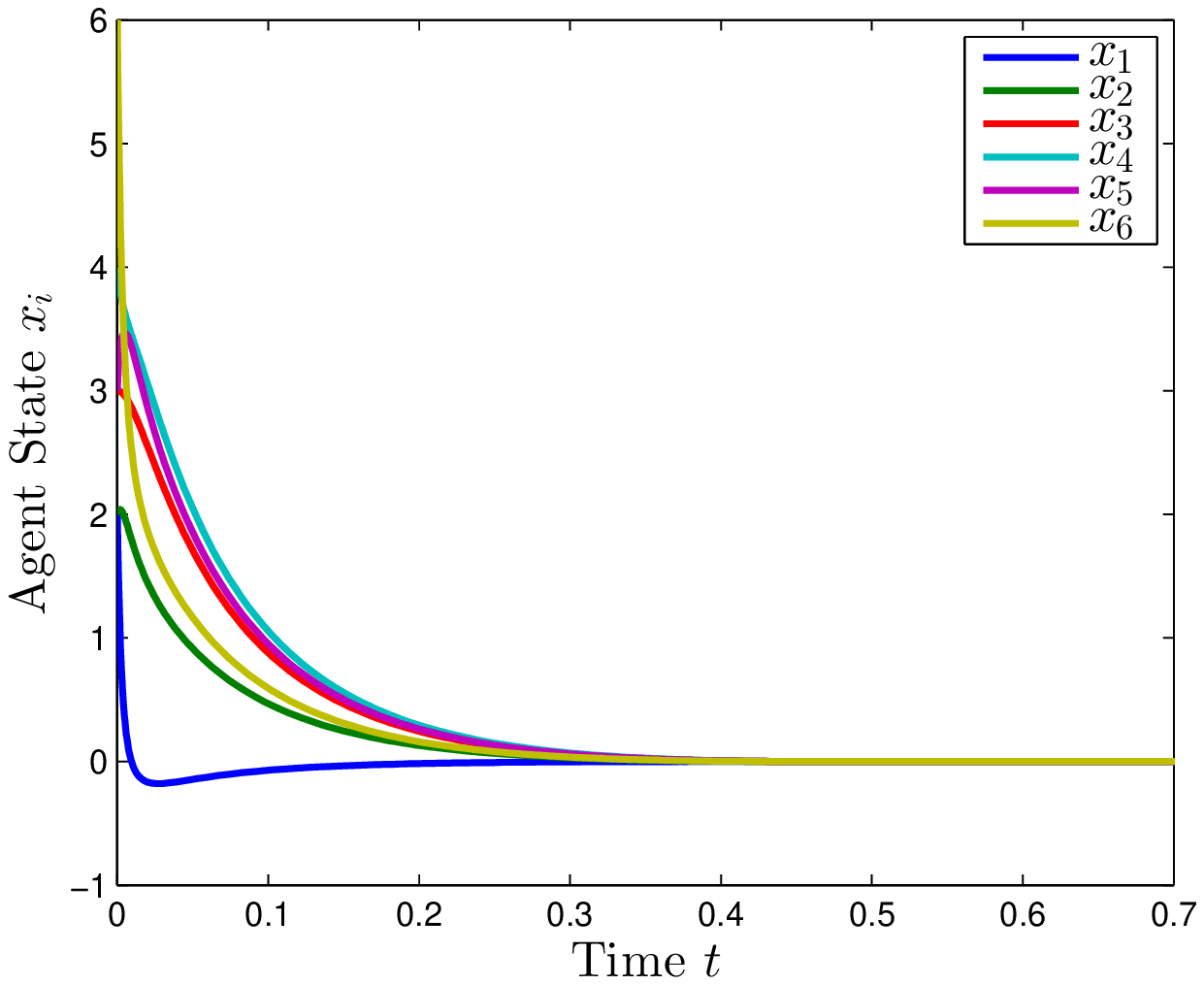}
\centering
\includegraphics[width=8cm]{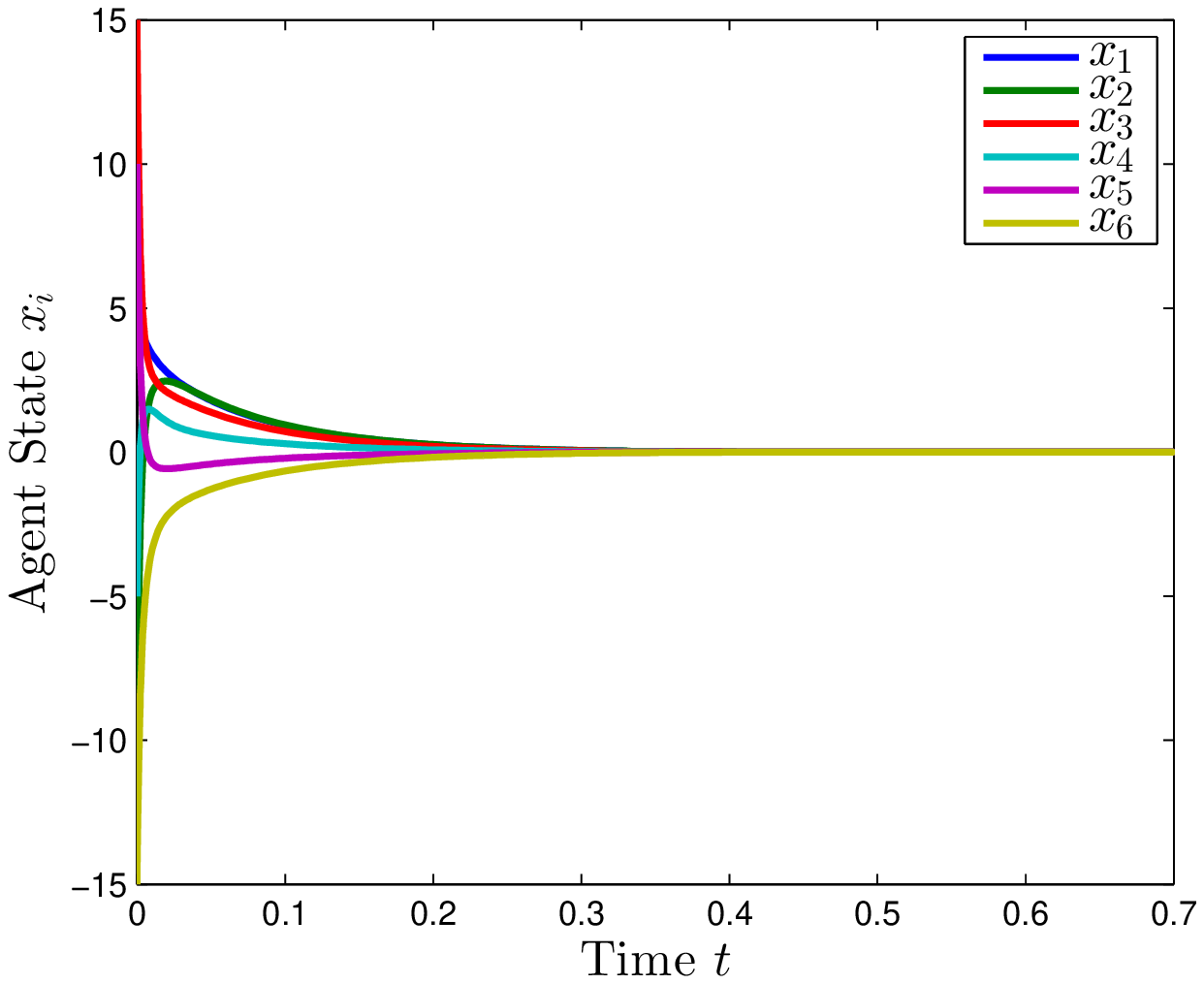}
\caption{\small Fixed-time state stability under the structurally unbalanced signed digraph $\mathcal{G}_b$. Upper: $x(0)=[2,2,3,4,3,6]^T$. Lower: $x(0)=[5,-10,15,-5,10,-15]^T$.}
\label{sim5}
\end{figure}
\end{example}

\section{Conclusions}

In this paper, we have investigated the distributed averaging problems of signed networks with general directed topologies. We have introduced the improved Laplacian potential function for signed networks, with which two distributed protocols are designed to make sure the signed-average consensus of signed networks no matter whether their associated signed digraphs are weight balanced or not. With improved Laplacian potential functions, we have given the convergence analyses for our two proposed protocols by employing the Lyapunov-based method, which gives a novel approach to exploring distributed control problems of directed signed networks. In addition, we have provided an application of improved Laplacian potential functions, in which the fixed-time bipartite consensus issues can be worked out for directed signed networks. Two simulation examples have been presented to validate the effectiveness of our developed results.

\section*{Appendix}

\subsection{Proof of Lemma \ref{lemma4}}\label{ap1}
\begin{proof}
Let $\hat{L}$ and $\hat{\Delta}$ $=$ $\mathrm{diag}\{\hat{\Delta}_{11}$, $\hat{\Delta}_{22}$, $\cdots$, $\hat{\Delta}_{nn}\}$ denote the Laplacian matrix and in-degree matrix of $\hat{\mathcal{G}}$, respectively. The diagonal element of $1/2(WL+L^TW)$ is given by
\begin{equation}\label{equ16}
\begin{split}
\tilde{\Delta}_{ii}=\frac{\det(\overline{L}_{ii})\sum_{j=1}^n|a_{ij}|+\det(\overline{L}_{ii})\sum_{j=1}^n|a_{ij}|}{2},~\forall i\in \mathcal{F}_n.
\end{split}
\end{equation}

\noindent Define a vector $w=[\det(\overline{L}_{11})$, $\det(\overline{L}_{22})$, $\cdots$, $\det(\overline{L}_{nn})]^T\in \mathbb{R}^{n}$ and $w$ satisfies $w^T\overline{L}=0_n^T$ from \cite{Li09}. We thus can induce
\begin{equation}\label{equ17}
\det(\overline{L}_{ii})\sum_{j=1}^n|a_{ij}|=\sum_{j=1}^n\det(\overline{L}_{jj})|a_{ji}|,~\forall i\in \mathcal{F}_n.
\end{equation}

\noindent With (\ref{equ17}), the equation (\ref{equ16}) can be rewritten as
\begin{equation*}
\begin{split}
\tilde{\Delta}_{ii}&=\frac{\det(\overline{L}_{ii})\sum_{j=1}^n|a_{ij}|+\sum_{j=1}^n\det(\overline{L}_{jj})|a_{ji}|}{2}\\
                   &=\hat{\Delta}_{ii}
\end{split}
\end{equation*}

\noindent Let $\tilde{\Delta}$ $=$ $\mathrm{diag}\{\tilde{\Delta}_{11}$, $\tilde{\Delta}_{22}$, $\cdots$, $\tilde{\Delta}_{nn}\}$. It can further derive
\begin{equation*}
\frac{1}{2}(WL+L^TW)=\tilde{\Delta}-\frac{WA+A^TW}{2}=\hat{\Delta}-\hat{A}=\hat{L}
\end{equation*}

\noindent which satisfies the definition of the Laplacian matrix. Hence, $1/2(WL+L^TW)$ is a valid Laplacian matrix of $\hat{\mathcal{G}}$.
\end{proof}

\subsection{Proof of Lemma \ref{lemma3}}
\begin{proof}
1): Since the signed digraph $\mathcal{G}$ is strongly connected, it follows from \cite{Li09} that all elements  $\det(\overline{L}_{11})$, $\det(\overline{L}_{22})$, $\cdots$, $\det(\overline{L}_{nn})$ are positive real numbers. This, together with (\ref{eq:7}), ensures
\begin{equation}
a_{ij}\neq 0\Rightarrow \hat{a}_{ij}\neq 0~\mbox{and}~ \hat{a}_{ji}\neq 0,~\forall i,j\in \mathcal{F}_n
\end{equation}

\noindent which causes that the mirror signed graph $\hat{\mathcal{G}}$ is connected due to the strong connectivity of $\mathcal{G}$.

2) According to the result 1) of Lemma \ref{lemma3}, the mirror signed graph $\hat{\mathcal{G}}$ of $\mathcal{G}$ is connected and undirected. With (\ref{eq:7}), we have
\begin{equation}\label{equ19}
\left\{\begin{split}
&a_{ij}>0~\mbox{or}~a_{ji}>0~\Leftrightarrow~\hat{a}_{ij}>0~\mbox{and}~ \hat{a}_{ji}>0\\
&a_{ij}<0~\mbox{or}~a_{ji}<0~\Leftrightarrow~\hat{a}_{ij}<0~\mbox{and}~ \hat{a}_{ji}<0
\end{split}
\right.,~\forall i,j\in \mathcal{F}_n.
\end{equation}

\noindent From \cite{Bein78}, we know that $\mathcal{G}$ is structurally balanced if and only if all semi-cycles of $\mathcal{G}$ are positive. By (\ref{equ19}), we can derive that all semi-cycles of $\hat{\mathcal{G}}$ are positive if and only if all semi-cycles of $\mathcal{G}$ are positive. Therefore, $\hat{\mathcal{G}}$ is structurally balanced if and only if $\mathcal{G}$ is structurally balanced.

Since structural balance and unbalance are mutually exclusive properties, we can directly develop that $\hat{\mathcal{G}}$ is structurally unbalanced if and only if $\mathcal{G}$ is structurally unbalanced.
\end{proof}

\subsection{Proof of Lemma \ref{lemma5}}
\begin{proof}
1) ``$\Leftarrow$": Based on \cite{Alta13}, we can develop $\mathcal{N}(L)=\mathrm{span}\{D_n1_n\}$ when $\mathcal{G}$ is structurally balanced, where $D_n\in \mathcal{D}$ satisfies $\overline{L}=D_nLD_n$. On one hand, for $\forall x\in\mathcal{N}(L)$, we have $x^T\hat{L}x=x^T\frac{WL+L^TW}{2}x=0_n$ that indicates $x\in \mathcal{N}\left(\hat{L}\right)$. We thus can obtain $\mathcal{N}(L)\subseteq \mathcal{N}\left(\hat{L}\right)$. On the other hand, since $\mathcal{G}$ is structurally balanced and strongly connected, a direct consequence of Lemma \ref{lemma3} is that $\hat{\mathcal{G}}$ is also structurally balanced and connected. It follows from Lemma \ref{lemma4} and \cite[Lemma 1]{Alta13} that $\mathrm{rank}\left(\hat{L}\right)=n-1$ holds. Hence, we can deduce $\mathcal{N}\left(\hat{L}\right)=\mathcal{N}(L)=\mathrm{span}\{D_n1_n\}$.

``$\Rightarrow$": Because of $\mathcal{N}\left(\hat{L}\right)=\mathcal{N}(L)=\mathrm{span}\{D_n1_n\}$, it can induce $LD_n1_n=0_n$ and $\mathrm{rank}(L)=n-1$. We can further derive $D_nLD_n1_n=\overline{L}1_n=0_n$ that implies $D_nLD_n=\overline{L}$. It is immediate to develop that $\mathcal{G}$ is structurally balanced from \cite[Theorem 4.1]{Meng18}.

2) ``$\Leftarrow$": Since $\mathcal{G}$ is structurally unbalanced, all eigenvalues of $L$ and $\hat{L}$ have positive real parts. Therefore, $\mathcal{N}\left(\hat{L}\right)=\mathcal{N}(L)=0_n$ holds.

``$\Rightarrow$": Owing to $\mathcal{N}\left(\hat{L}\right)=\mathcal{N}(L)=0_n$, we can realize that the equation $Lx=0_n$ have no non-zero solution and thus obtain $\mathrm{rank}(L)=n$. It follows from \cite[Corollary 3]{Alta13} that $\mathcal{G}$ is structurally unbalanced.
\end{proof}

\subsection{Proof of Lemma \ref{lem3}}

\begin{proof}
For the first statement, we can calculate
\begin{equation}\label{eq6}
\begin{split}
x^TWLx&=\sum_{i=1}^n\det(\overline{L}_{ii})x_i\sum_{j=1}^n|a_{ij}|(x_i-\mathrm{sgn}(a_{ij})x_j)\\
&=\sum_{i=1}^n\sum_{j=1}^n\det(\overline{L}_{ii})|a_{ij}|x_i(x_i-\mathrm{sgn}(a_{ij})x_j).
\end{split}
\end{equation}

\noindent By employing (\ref{equ17}), we can further deduce
\begin{equation}\label{eq7}
\begin{split}
x^TWLx&=\sum_{i=1}^n\det(\overline{L}_{ii})x_i\sum_{j=1}^n|a_{ij}|(x_i-\mathrm{sgn}(a_{ij})x_j)\\
&=\sum_{i=1}^n\det(\overline{L}_{ii})x_i^2\sum_{j=1}^n|a_{ij}|\\
&~~~~~~~~~~~-\sum_{i=1}^n\sum_{j=1}^n\det(\overline{L}_{ii})|a_{ij}|\mathrm{sgn}(a_{ij})x_ix_j\\
&=\sum_{i=1}^nx_i^2\sum_{j=1}^n\det(\overline{L}_{jj})|a_{ji}|\\
&~~~~~~~~~~~-\sum_{i=1}^n\sum_{j=1}^n\det(\overline{L}_{ii})|a_{ij}|\mathrm{sgn}(a_{ij})x_ix_j\\
&=\sum_{j=1}^nx_j^2\sum_{i=1}^n\det(\overline{L}_{ii})|a_{ij}|\\
&~~~~~~~~~~~-\sum_{i=1}^n\sum_{j=1}^n\det(\overline{L}_{ii})|a_{ij}|\mathrm{sgn}(a_{ij})x_ix_j\\
&=\sum_{i=1}^n\sum_{j=1}^n\det(\overline{L}_{ii})|a_{ij}|x_j(x_j-\mathrm{sgn}(a_{ij})x_i).
\end{split}
\end{equation}

\noindent With (\ref{eq6}) and (\ref{eq7}), we have
\begin{equation*}
\begin{split}
x^T&(WL+L^TW)x=2x^TWLx\\
&=\left\{\sum_{i=1}^n\sum_{j=1}^n\det(\overline{L}_{ii})|a_{ij}|x_i(x_i-\mathrm{sgn}(a_{ij})x_j)\right.\\
&~~~\left.+\sum_{i=1}^n\sum_{j=1}^n\det(\overline{L}_{ii})|a_{ij}|x_j(x_j-\mathrm{sgn}(a_{ij})x_i)\right\}\\
&=\sum_{i=1}^n\sum_{j=1}^{n}\det(\overline{L}_{ii})|a_{ij}|(x_i-\mathrm{sgn}(a_{ij})x_j)^2=\Phi_\mathrm{e}(x).
\end{split}
\end{equation*}

\noindent Hence, the equation (\ref{equ:5}) holds. Based on (\ref{equ:5}), the remaining of this proof can be immediately derived from Lemmas 1-3.
\end{proof}

\end{document}